\documentclass[12pt]{amsart}

\usepackage{amsmath, amsthm, amssymb}
\input xy
\xyoption{all}
\usepackage{mathrsfs}
\usepackage{enumerate}
\usepackage{color}
\usepackage{tikz-cd}
\usetikzlibrary{matrix}
\usepackage{caption}
\usepackage{subcaption}
\usepackage{mathtools}
\usepackage{geometry}
\usepackage{cite} 
\newtheorem{dummy}{}[section]
\newtheorem{theorem}[dummy]{Theorem}
\newtheorem{proposition}[dummy]{Proposition}
\newtheorem{lemma}[dummy]{Lemma}
\newtheorem{corollary}[dummy]{Corollary}

\theoremstyle{definition}
\newtheorem{definition}[dummy]{Definition}

\newtheorem{remark}[dummy]{Remark}
\newtheorem{convention}[dummy]{Convention}

\newcommand{\vir}{\mathrm{vir}}

\newcommand{\mb}{\mathbf}
\newcommand{\bb}{\mathbb}
\newcommand{\ca}{\mathcal}

\newcommand{\sslash}{\mathbin{/\mkern-6mu/}}

\title{The Level Structure in Quantum $K$-theory and Mock Theta Functions}

\author[Y.~Ruan]{Yongbin Ruan}
\address{Department of Mathematics, University of Michigan, 2074 East Hall, 530 Church Street, Ann Arbor, MI 48109, USA}
\email{ruan@umich.edu}
\author[M.~Zhang]{Ming Zhang}
\address{Department of Mathematics, University of Michigan, 2074 East Hall, 530 Church Street, Ann Arbor, MI 48109, USA}

\email{zhangmsq@umich.edu}

\begin{document}

 
\maketitle

\begin{abstract}
This is the first in a sequence of papers to develop the theory of levels in quantum $K$-theory and study its applications. In this paper, we give an adelic characterization of the range of the $J$-function in quantum $K$-theory with level structure. As an application, we prove a mirror theorem for permutation-equivariant quantum $K$-theory with level structures of toric varieties. In the study of the mirrors of some simple examples, we see the surprising appearance of Ramanujan's mock theta functions.\end{abstract}

\section{Introduction}
\subsection{Overview}
More than a decade ago, quantum $K$-theory was introduced by Givental and Lee \cite{givental3,lee1} as the
$K$-theoretic analog of quantum cohomology. 
Its recent revival stems partially from a physical interpretation of quantum 
$K$-theory as a 3D-quantum field theory in the 3-manifold
of the form $S^1\times \Sigma$ (see \cite{willett,jockers}). Because of this mysterious physical connection, the B-model counterpart of quantum $K$-theory is $q$-hypergeometric series, itself a classical subject. The above connection was recently confirmed by Givental \cite{givental16} as the mirror of the so-called \emph{$J$-function} of the permutation-equivariant
quantum $K$-theory. 

Classically, $K$-theory is more closely related to representation theory, comparing to cohomology theory. It is natural to revisit quantum $K$-theory from the representation theoretic point of view. In fact, a variant of quantum $K$-theory was
already studied by Aganagic, Okounkov, Smirnov and their collaborators \cite{okounkov,okounkov2,okounkov3,smirnov1,smirnov2,smirnov3,su} in relation to quantum groups. One of the
predominant features of representation theory is the existence of an additional parameter 
called the level. A natural question is whether it is possible to extend the current version of quantum
$K$-theory to include this notion of level. In this paper, we answer the question affirmatively
in the context of quasimap theory. This is the first in a sequence of papers to develop the theory of levels in
quantum $K$-theory and study its applications.

Our motivating example is an old physical result of Witten \cite{witten1} in the early '90s which 
relates the quantum cohomology ring of the Grassmannian to the Verlinde algebra. Early explicit physical computations \cite{gepner1,vafa,intriligator} indicate that they are isomorphic as algebras, but have different pairings. In \cite{witten1}, Witten gave a conceptual explanation of the isomorphism, by proposing an equivalence between 
the quantum field theories which govern the Verlinde algebra and the quantum cohomology of the Grassmannian. His physical derivation of the equivalence naturally leads to a mathematical problem that
these two objects are conceptually isomorphic (without referring to the detailed computations). A great deal of work has been done by Agnihotri \cite{agnihotri}, Marian-Oprea \cite{marian1,marian2,marian3}, and Belkale \cite{belkale}. However, to the best of our knowledge, a complete conceptual proof of the equivalence is missing.

Assuming a basic knowledge of quantum $K$-theory, a key and yet more or less trivial observation is  that Verlinde algebra is a \emph{K-theoretic invariant}.
To be more precise, let $X$ be a smooth projective variety. Suppose that $\overline{\ca{M}}_{g,k}(X, \beta)$
is the moduli space of stable maps to $X$. Quantum cohomology studies integrals of the form
$$\int_{[\overline{\ca{M}}_{g,k}(X, \beta)]^{\vir}}\alpha,$$
where 
$[\overline{\ca{M}}_{g,k}(X, \beta)]^{\vir}$ is the so-called \emph{virtual fundamental cycle} and $\alpha$ is a ``tautological" cohomology class. In quantum $K$-theory, we replace the virtual fundamental cycle by the \emph{virtual structure sheaf} ${\ca{O}}^{\vir}_{\overline{\ca{M}}_{g,k}(X, \beta)}$. We also replace the integral by the holomorphic Euler characteristic 
$$\chi\big(\overline{\ca{M}}_{g,k}(X, \beta), E\otimes{\ca{O}}^{\vir}_{\overline{\ca{M}}_{g,k}(X, \beta)}\big)$$
where $E$ is some natural $K$-theory class on $\overline{\ca{M}}_{g,k}(X, \beta)$. For the Verlinde algebra, the relevant moduli space is the moduli space of semistable parabolic $U(n)$-bundles ${\ca{M}}_{U(n)}(\alpha_1,\cdots, \alpha_k)$
on a fixed genus $g$ marked curve $(C, p_1,\cdots, p_k)$ with parabolic structure at $p_i$  indexed by $\alpha_i\in V_l(U(n))$. Here, $l$ is a non-negative integer and $V_l(U(n))$ denotes the level-$l$ Verlinde algebra. A new ingredient is a certain determinant line bundle, denoted by $\det$, over the moduli space ${\ca{M}}_{U(n)}(\alpha_1,\cdots, \alpha_k)$. The level-$l$ Verlinde algebra
calculates the holomorphic Euler characteristic
$$\langle \alpha_1, \cdots, \alpha_k\rangle^{l}_{g, \text{Verlinde}}=\chi({\ca{M}}_{U(n)}(\alpha_1, \cdots, \alpha_k), \text{det}^l).$$

Based on the above description, the Verlinde algebra is clearly a $K$-theoretic object, and we should compare it with the quantum $K$-theory of the Grassmannian (with an appropriate notion of levels). 
 Let $\mathfrak{B}un_G$ be the moduli stack of principal bundle over curves. Let $\pi:\mathfrak{C}_{\mathfrak{B}un_{g,k}}\rightarrow\mathfrak{B}un_G$ be the universal curve and let $\mathfrak{P}\rightarrow\mathfrak{C}_{\mathfrak{B}un_{g,k}}$ be the universal principal bundle. Given a finite-dimensional representation $R$ of $G$, we consider the inverse determinant of cohomology 
\[
\text{det}_R:=\big(\text{det}\,R\pi_*(\mathfrak{P}\times_G R)\big)^{-1}.
\] 
It is a line bundle over $\mathfrak{B}un_G$. Suppose $X=Z\sslash G$ is a GIT quotient. Let ${\ca{Q}}^{\epsilon}$ be the moduli stack of $\epsilon$-stable quasimaps to $X=Z\sslash G$ (see Section \ref{quasimapthy}). Then there is a natural forgetful morphism $\mu:\ca{Q}^{\epsilon}\rightarrow \mathfrak{Bun}_G$. We define the level-$l$ determinant line bundle as
$${\ca{D}}^{R,l}=\mu^*(\text{det}_R)^l.$$ 
We will often refer to $\ca{D}^{R,l}$ as the \emph{level structure}. In general, when $X$ is a smooth complex projective variety, but not a GIT quotient, one can still define determinant line bundles over the moduli space of stable maps $\overline{\ca{M}}_{g,k}(X,\beta)$ as follows. Let $\ca{R}$ be a vector bundle over $X$. Let $\pi:\ca{C}\rightarrow\overline{\ca{M}}_{g,k}(X,\beta)$ be the universal curve and let $\text{ev}:\ca{C}\rightarrow X$ be the universal evaluation map. We define the level-$l$ determinant line bundle as
\[
{\ca{D}}^{l}:=\big(\text{det}\,R\pi_*(\text{ev}^*\ca{R})\big)^{-l}.
\]
This definition agrees with the previous one when $X$ is a GIT quotient (see Definition \ref{defdet1} and Remark \ref{defdet2}). 

With the above definition of the level-$l$ determinant line bundle $\ca{D}^{R,l}$, we can define the level-$l$ quantum $K$-invariants and quasimap invariants by twisting with $\ca{D}^{R,l}$ (see Section \ref{quantumk}). The ordinary quantum $K$-theory corresponds to the case $l=0$. One of main purpose of this article is to verify that 
the level-$l$ quantum $K$-theory satisfies all the axioms of the ordinary quantum $K$-theory (see Section \ref{property}). 

One can use  irreducible $\text{GL}_n(\bb{C})$-representations to define two sets of invariants: Verlinde invariants from the theory of (semi)stable parabolic vector bundles and level-$l$ quantum K-theory invariants of the Grassmannian $\text{Gr}(n, n+l)$. The \emph{K-theoretic Verlinde/Grassmannian correspondence} describes a conjectural formula between these two sets of invariants (see the precise formula in \cite{RZ1}).

\subsection{Mirror theorem and mock theta functions}
The proof of the Verlinde/Grassmannian correspondence will be discussed in different papers. The other main results of this paper are  various mirror theorems, in the same style as the recent work of Givental \cite{givental11,givental12,givental13,givental14,givental15,givental16,givental17,givental18,givental19,givental21,givental22}. In the study of quantum $K$-theory with level structures, we see the surprising appearance of Ramanujan's mock theta functions in some of the simplest examples. 

Let $X$ be a smooth complex projective variety and let $\ca{R}$ be a vector bundle over $X$. When $X=Z\sslash G$ is a GIT quotient, we assume $\ca{R}$ is of the form $(Z\times R)\sslash G$, with $R$ a finite-dimensional representation of $G$. By abuse of terminology, we refer to the vector bundle $\ca{R}$ as the ``representation'' $R$. Let $Q$ be the Novikov variables. We fix a $\lambda$-algebra $\Lambda$ which is equipped with Adams operations $\Psi^i,\,i=1,2,\dots$. Let $\{\phi_a\}$ be a basis of $K^0(X)\otimes\bb{Q}$ and let $\{\phi^a\}$ be the dual basis with respect to the twisted pairing (\ref{eq:twistedpairing}). Let $q$ be a formal variable and let $\mb{t}(q)$ be a Laurent polynomial in $q$ with coefficients in $K^0(X)\otimes\bb{Q}$. The \emph{permutation-equivariant} $K$-theoretic $J$-function $\ca{J}^{l}_{S_\infty}(\mb{t}(q),Q)$ of level $l$ and representation $R$ is defined by
\[
\ca{J}_{S_\infty}^{R,l}(\mb{t}(q),Q):=1-q+\mb{t}(q)+\sum_a\sum_{\beta\neq 0}Q^\beta\phi^a\bigg\langle\frac{\phi_a}{1-qL},\mb{t}(L),\dots,\mb{t}(L)\bigg\rangle_{0,k+1,\beta}^{R,l,S_k}.
\]
Here $\langle\,\cdot\,\rangle_{0,k+1,\beta}^{R,l,S_k}$ denotes the permutation-equivariant quantum $K$-invariants of level $l$ and $L$ denote the cotangent line bundles. The $J$-function $\ca{J}_{S_\infty}^{R,l}$ can be viewed as elements in the loop space $
\ca{K}$ defined by
\[
\ca{K}:=[K^0(X)\otimes\bb{C}(q)]\otimes\bb{C}[[Q]],
\]
where $\bb{C}(q)$ denotes the field of complex rational functions in $q$. There is a natural Lagrangian polarization  $\ca{K}=\ca{K}_+\oplus\ca{K}_-$, where $\ca{K}_+$ consists of Laurent polynomials and $\ca{K}_-$ consists of reduced rational functions regular
at $q = 0$ and vanishing at $q =\infty$. We denote by $\ca{L}^{R,l}_{S_\infty}$ the range of $\ca{J}_{S_\infty}^{R,l}$, i.e., $$\ca{L}_{S_\infty}^{R,l}=\cup_{\mb{t}(q)\in\ca{K}_+}\ca{J}_{S_\infty}^{R,l}(\mb{t}(q),Q)\subset\ca{K}.$$

Due to the stacky structure of the moduli space of stable maps, the $J$-function $\ca{J}_{S_\infty}^{R,l}(\mb{t}(q),Q)$, as a function in $q$, has poles at roots of unity. The main technical tool is a generalization of Givental-Tonita's \emph{adelic characterization} of points on the cone $\ca{L}^{R,l}_{S_\infty}$ with the presence of level structure, i.e., we describe the Laurent expansion of $\ca{J}_{S_\infty}^{R,l}$ at each primitive root of unity in terms of certain twisted \emph{fake} quantum $K$-theory. The precise statement is rather technical, and we present it in Theorem \ref{levelladele}. When $l=0$, the determinant line bundle $\ca{D}^{R,l}$ is trivial, and we recover the ordinary quantum $K$-theory. The adelic characterization of the cone in the ordinary quantum $K$-theory was introduced in \cite{givental2}, and its generalization to the permutation-equivariant theory is given in \cite{givental13}. The proofs of all these results are based on application of the virtual Kawasaki's Riemann-Roch formula to moduli spaces of stable maps (see Section \ref{conechar}).

Let $\ca{L}_{S_\infty}$ denote the range of the permutation-equivariant big $J$-function in ordinary quantum $K$-theory (i.e., with trivial level structure). As an application of Theorem \ref{levelladele}, we prove that certain ``determinantal '' modifications of points on $\ca{L}_{S_\infty}$ lie on the cone $\ca{L}^{R,l}_{S_\infty}$ of quantum $K$-theory of level $l$.
\begin{theorem}\label{detmodify}
If \[
I=\sum_{\beta\in\emph{Eff}(X)}I_\beta Q^\beta
\]
lies on $\ca{L}_{S_\infty}$, then the point
\[
I^{R,l}:=\sum_{\beta\in\emph{Eff}(X)}I_\beta Q^\beta\prod_i\big(L_i^{-\beta_i}q^{(\beta_i+1)\beta_i/2}\big)^l
\]
lies on the cone $\ca{L}_{S_\infty}^{R,l}$ of permutation-equivariant quantum $K$-theory of level $l$. Here, $\emph{Eff}(X)$ denotes the semigroup of effective curve classes on $X$, $L_i$ are the $K$-theoretic Chern roots of $\ca{R}$, and $\beta_i:=\int_\beta\,c_1(L_i)$.
\end{theorem}

In Theorem \ref{ifunc}, we give explicit formulas for level-$l$ (torus-equivariant) $I$-functions of toric varieties. Moreover, we prove the following toric mirror theorem.
\begin{theorem}\label{weakmirror}
Assume that $X$ is a smooth quasi-projective toric variety. The level-$l$ torus-equivariant $I$-function $(1-q)I^{R,l}$ of $X$ lies on the cone $\ca{L}_{S_\infty}^{R,l}$ in the permutation- and torus-equivariant quantum K-theory of level $l$ of $X$.
\end{theorem}

In the study of toric mirror theorems for quantum $K$-theory with level structure, a remarkable phenomenon is the appearance of Ramanujan's mock theta
functions. We first establish some notations. Denote the standard representation of $G$ by St and its dual by $\text{St}^\vee$. When $G=\bb{C}^*$, any $n$-dimensional representation of $G$ is determined by a \emph{charge vector} $(a_1,\dots,a_n)$ with $a_i\in\bb{Z}$: a $\bb{C}^*$-action on $\bb{C}^n$ can be explicitly described by
\[
\lambda\cdot(x_1,\dots,x_n)\rightarrow(\lambda^{a_1}x_1,\dots,\lambda^{a_n}x_n),\quad\text{where}\ \lambda\in\bb{C}^*
.\]
In the following propositions, we consider GIT quotients $\bb{C}^n\sslash\bb{C}^*$, and we refer to the $\bb{C}^*$-actions by their associated charge vectors.

\begin{proposition}\label{prop1}
Consider $X=\bb{C}\sslash\bb{C}^*=[({\bb{C}}\backslash0)/{\bb{C}}^*]$, where the $\bb{C}^*$-action is the standard action by multiplication. The $\bb{C}^*$-equivariant $K$-ring $K_{\bb{C}^*}(X)$ is isomorphic to the representation ring $\emph{Repr}(\bb{C}^*)$. Let $\lambda\in K_{\bb{C}^*}(X)$ be the equivariant parameter corresponding to the standard representation. For the $\bb{C}^*$-representations $\emph{St}$ and $\emph{St}^\vee$, we have the following explicit formulas of the equivariant small $I$-functions
\begin{align*}
I^{\emph{St},\,l}_{X}(q,Q)&=1+\sum_{n\geq 1}\frac{q^{\frac{n(n-1)l}{2}}}{(1-\lambda^{-1} q)(1-\lambda^{-1} q^2)\cdots (1-\lambda^{-1} q^n)}Q^n,\\
I^{\emph{St}^\vee,\,l}_{X}(q,Q)&=1+\sum_{n\geq 1}\frac{q^{\frac{n(n+1)l}{2}}}{(1-\lambda^{-1} q)(1-\lambda^{-1} q^2)\cdots (1-\lambda^{-1} q^n)}Q^n,\\
\end{align*}
By chosing certain specializations of the parameters, we obtain Ramanujan's mock theta functions of order 3
$$I^{\emph{St},\,l=1}_X(q^2,Q)|_{\lambda=-1, Q=q}=1+\sum_{n\geq 1}\frac{q^{n^2}}{(1+q^2)(1+q^4)\cdots (1+q^{2n})},$$
$$I^{\emph{St},\,l=1}_X(q^2,Q)|_{\lambda=q, Q=q}=1+\sum_{n\geq 1}\frac{q^{n^2}}{(1-q)(1-q^3)\cdots (1-q^{2n-1})},$$
$$I^{\emph{St},\,l=1}_X(q^2,Q)|_{\lambda=-q, Q=1}=1+\sum_{n\geq 1}\frac{q^{n(n-1)}}{(1+q)(1+q^3)\cdots (1+q^{2n-1})},$$
and Ramanujan's mock theta functions of order 5
$$I^{\emph{St},\,l=2}_X(q,Q)|_{\lambda=-1, Q=q}=1+\sum_{n\geq 1}\frac{q^{n^2}}{(1+q)(1+q^2)\cdots (1+q^n)},$$
$$I^{\emph{St},\,l=2}_X(q^2,Q)|_{ \lambda=q, Q=q^2}=1+\sum_{n\geq 1}\frac{q^{2n^2}}{(1-q)(1-q^3)\cdots (1-q^{2n-1})},$$
$$I^{\emph{St}^\vee,\,l=2}_X(q,Q)|_{\lambda=-1, Q=1}=1+\sum_{n\geq 1}\frac{q^{n(n+1)}}{(1+q)(1+q^2)\cdots (1+q^{n})},$$
$$I^{\emph{St}^\vee,\,l=2}_X(q^2,Q)|_{\lambda=q, Q=1}=1+\sum_{n\geq 1}\frac{q^{2n^2+2n}}{(1-q)(1-q^3)\cdots (1-q^{2n-1})}.$$
\end{proposition}

\begin{proposition}\label{prop2}
Let $a_1$ and $a_2$ be two positive integers which are coprime. We consider the target $X_{a_1,a_2}=[({\bb{C}}^2\backslash0)/{\bb{C}}^*]$ with charge vector $(a_1,a_2)$ and a line bundle $p=[\{({\bb{C}}^2\backslash0)\times\bb{C}\}/{\bb{C}}^*]$ with charge vector $(a_1,a_2,1)$. Let $\lambda_1$ and $\lambda_2$ be the equivariant parameters. For the $\bb{C}^*$-representations $\emph{St}$ and $\emph{St}^\vee$, we have the following explicit formulas of the equivariant small $I$-functions
\begin{align*}
I^{\emph{St},\,l}_{X_{a_1,a_2}}(q,Q)&=1+\sum_{n\geq 1}\frac{ p^{nl}q^{\frac{n(n-1)l}{2}}}{(1-p^{a_1}\lambda_1^{-1} q)\cdots (1-p^{a_1}\lambda_1^{-1} q^{a_1n})(1-p^{a_2}\lambda_2^{-1} q)\cdots (1-p^{a_2}\lambda_2^{-1} q^{a_2n})}Q^n,\\
I^{\emph{St}^\vee,\,l}_{X_{a_1,a_2}}(q,Q)&=1+\sum_{n\geq 1}\frac{ p^{nl}q^{\frac{n(n+1)l}{2}}}{(1-p^{a_1}\lambda_1^{-1} q)\cdots (1-p^{a_1}\lambda_1^{-1} q^{a_1n})(1-p^{a_2}\lambda_2^{-1} q)\cdots (1-p^{a_2}\lambda_2^{-1} q^{a_2n})}Q^n.
\end{align*}
By choosing $(a_1,a_2)=(1,1)$ and certain specializations of the parameters, we obtain the following four Ramanujan's mock theta functions of order 3:
$$I^{\emph{St},\,l=2}_{X_{1,1}}(q^2,Q)|_{p=1,\lambda_1=\lambda_2=-1, Q=q}=1+\sum_{n\geq 1}\frac{q^{n^2}}{((1+q)(1+q^2)\cdots (1+q^n))^2},$$
$$I^{\emph{St},\,l=2}_{X_{1,1}}(q,Q)|_{p=1, \lambda_1=\frac{1+\sqrt{3}i}{2}, \lambda_2=\frac{1-\sqrt{3}i}{2}, Q=q}=1+\sum_{n\geq 1}\frac{q^{n^2}}{(1-q+q^2)(1-q^2+q^4)\cdots (1-q^n+q^{2n})},$$
$$\frac{1}{(1-q)^2}I^{\emph{St}^\vee,\,l=2}_{X_{1,1}}(q^2,Q)|_{p=1, \lambda_1=\lambda_2=q^{-1}, Q=1}=\sum_{n\geq 0}\frac{q^{2n^2+2n}}{((1-q)(1-q^3)\cdots (1-q^{2n+1}))^2},$$
\begin{align*}\frac{1}{(1+q+q^2)}I^{\emph{St}^\vee,\,l=2}_{X_{1,1}}(q^2,Q)&|_{p=1, \lambda_1=\frac{-1+\sqrt{3}i}{2}q^{-1}, \lambda_2=\frac{-1-\sqrt{3}i}{2}q^{-1}, Q=1}\\
&=\sum_{n\geq 0}\frac{q^{2n^2+2n}}{(1+q+q^2)(1+q^3+q^6)\cdots (1+q^{2n+1}+q^{4n+2})}.
\end{align*}
\end{proposition}

\begin{proposition}\label{prop3}
Let $a$ and $b$ be two positive integers which are coprime. We consider the target $X_{a,-b}=[\{({\bb{C}}\backslash0)\times {\bb{C}}\}/{\bb{C}}^*]$ with charge vector $(a,-b)$ and a line bundle $p= [\{({\bb{C}}\backslash0)\times {\bb{C}}\times\bb{C}\}/{\bb{C}}^*]$ with charge vector $(a,-b,1)$. Let $\lambda$ and $\mu$ be the equivariant parameters of the standard $(\bb{C}^*)^2$-action on $X_{a,-b}$. For the $\bb{C}^*$-representation $\emph{St}$, we have the following explicit formula for the equivariant small $I$-function
\[I^{\emph{St},\, l}_{X_{a,-b}}(q)=1+\sum_{n\geq 1}(-1)^{bn}\frac{ p^{nl-b^2n}q^{\frac{n(n-1)l-bn(bn-1)}{2}}\mu^{-bn}(1-p^{b}\mu)(1-p^b\mu  q)\cdots (1-p^b\mu  q^{bn-1})}{(1-p^a\lambda^{-1}  q)(1-p^a\lambda^{-1} q^2)\cdots (1-p^a\lambda^{-1}  q^{an})}Q^n.
\]
In particular, we have order 7 mock theta functions
$$I^{\emph{St},\,l=3}_{X_{2,-1}}(q,Q)|_{p=1, \lambda=1,\mu=q, Q=-q^2}=1+\sum_{n\geq 1}\frac{q^{n^2}}{(1-q^{n+1})\cdots (1-q^{2n})},$$
$$\frac{q}{1-q}I^{\emph{St},\,l=3}_{X_{2,-1}}(q,Q)|_{p=1, \lambda=q^{-1},\mu=q, Q=-q^4}=\sum_{n\geq 1}\frac{q^{n^2}}{(1-q^{n})\cdots (1-q^{2n-1})},$$
$$\frac{1}{1-q}I^{\emph{St},\,l=3}_{X_{2,-1}}(q,Q)|_{p=1, \lambda=q^{-1},\mu=q, Q=-q^3}=\sum_{n\geq 1}\frac{q^{n^2-n}}{(1-q^{n})\cdots (1-q^{2n-1})}.$$
\end{proposition}

\begin{remark}\label{firstorb}
The targets that we consider in Proposition \ref{prop2} and Proposition \ref{prop3} are in general orbifolds. The full $I$-functions have components corresponding to the twisted sectors of the (rigidified) inertia stacks of the orbifold targets. However, to match with mock theta functions, we only consider the components of $I$-functions corresponding to the untwisted sectors. See the discussion on the orbifold case in Remark \ref{orbmap1} and Remark \ref{orbmap2}. \end{remark}

It is interesting that we can recover Ramanujan's mock theta functions using only very simple targets. 

One of the attractive features of quantum $K$-theory is the appearance of $q$-hypergeometric series as mirrors of $K$-theoretic $J$-functions. Recall the definition of the $q$-Pochhammer symbol
$$(a;q)_n:=(1-a)(1-aq)\cdots (1-aq^{n-1})\quad\text{for}\ n>0,$$
and $(a;q)_0:=1$. A general $q$-hypergeometric series can be written as
$$_r \phi_s=\sum_{n\geq 0}\frac{(\alpha_1; q)_n\cdots (\alpha_r; q)_n}{(\beta_1; q)_n\cdots (\beta_s; q)_n}\frac{z^n}{(q; q)_n}[(-1)^n q^\frac{n(n-1)}{2}]^{1+s-r}.$$
For the quantum $K$-theory of level 0, i.e, Givental-Lee's quantum $K$-theory, we only see special $q$-hypergeometric series of the form
$$\sum_{n\geq 0}\frac{(\alpha_1; q)_n\cdots (\alpha_r; q)_n}{(\beta_1; q)_n\cdots (\beta_s; q)_n}\frac{z^n}{(q; q)_n}.$$
The level structure naturally introduces the term 
$$[(-1)^n q^\frac{n(n-1)}{2}]^{1+s-r}.$$

\begin{proposition}\label{prop4}
Consider the target $X_{\bf 1, -\bf 1}:=O(-1)^{\oplus r}_{\bb{P}^{s-1}}=[\{({\bb{C}}^{s}\backslash0)\times {\bb{C}}^r\}/{\bb{C}}^*]$
with the charge vector $(1,1,\cdots, 1, -1, -1, \cdots, -1)$. Let $p=[\{({\bb{C}}^{s}\backslash0)\times {\bb{C}}^r\times\bb{C}\}/{\bb{C}}^*]$ be a line bundle with charge vector $(1,\cdots, 1, -1,, \cdots, -1,1)$. Let $\lambda_1, \cdots, \lambda_s, \mu_1, \cdots, \mu_r$ be the equivariant parameters of the standard $({\bb{C}}^*)^{s+r}$-action on $X_{\bf 1, -1}$. Then the equivariant small
$I$-function has the following explicit form
$$I^{\emph{St},\,l=1+s}_{X_{\bf 1, -\bf 1}}(q)=1+\sum_{n\geq1}(-1)^{nr}\prod_{i=1}^r(p\mu_i)^{-n}p^{(1+s)n}\frac{(p\mu_1, q)_{n}\cdots (p\mu_r; q)_{n}}
{(p\lambda_1^{-1}q; q)_{n}\cdots (p\lambda_s^{-1}q; q)_n}Q^n (q^{\frac{n(n-1)}{2}})^{1+s-r}.$$
Hence we can recover the general $q$-hypergeometric series by setting $p=1,\lambda_i^{-1}q=\beta_i,\mu_j=\alpha_j, Q= (-1)^{1+s}z\prod_{i=1}^r\mu_i$.
\end{proposition}

Recall that Gromov-Witten theory (of Calabi-Yau varieties) is related to quasi-modular forms. Mock modular forms are another class of modular objects, which are different from the quasi-modular forms. Yet, they share some common properties. The above mirror theorems suggest an exciting possibility that
the natural geometric home of mock modular forms is quantum $K$-theory with non-trivial level structures. We certainly would like to investigate
this surprising connection further.

\subsection{Plan of the paper}
In Section \ref{level}, we introduce the notion of level in the $K$-theoretic quasimap theory and establish its main properties. In Section \ref{quantumk}, we define the $K$-theoretic quasimap invariants with level structure and their permutation-equivariant version. In Section \ref{conechar}, we characterize the cone $\ca{L}_{S_\infty}^{R,l}$ of level $l$ in terms of the cone $\ca{L}_{S_\infty}$ of level 0. In Section \ref{mirrorandmock}, we compute the $K$-theoretic toric $I$-functions of level $l$, and prove a toric mirror theorem for quantum $K$-theory with level structure.
\subsection{Acknowledgments}

The first author would like to thank Aaron Bertram, Prakash Belkale and Jie Zhou for interesting discussions. 
The second author would like to thank Felix Janda, Feng Qu, Valentin Tonita, Rachel Webb, Yaoxiong Wen and Fenglong You for helpful discussions.

The first author is partially supported by the NSF grant DMS 1405245 and the NSF FRG grant DMS 1159265.

\section{ Level structure}\label{level}
In this section, we assume, unless otherwise indicated, that $X$ is given by a geometric invariant theory (GIT) quotient. As mentioned in the introduction, the level structure is defined by determinant line bundles. We first recall the definition of determinant line bundles and then define the notion of level in quasimap theory.

\subsection{Determinant line bundles}
In this subsection, we briefly review the construction of determinant line bundles. 

Let $\ca{X}$ be a Deligne-Mumford stack. Let $\ca{E}$ be a locally free, finitely generated $\ca{O}_{\ca{X}}$ module. We define the determinant line bundle of $\ca{E}$ as 
\[
\text{det}(\ca{E}):=\wedge^{\text{rank}(\ca{E})}\ca{E},
\] 
where $\wedge^i$ denotes the $i$-th wedge product.
In general, let $\ca{F}^\bullet$ be a complex of coherent sheaves on $\ca{X}$ which has a bounded locally free resolution, i.e., there exists a bounded complex of locally free, finitely generated $\ca{O}_{\ca{X}}$ modules $\ca{G}^\bullet$ and a quasi-isomorphism
\[
\ca{G}^\bullet\rightarrow\ca{F}^\bullet.
\]
We define the determinant line bundle associated to $\ca{F}^\bullet$ by
\[
\text{det}(\ca{F}^\bullet):=\otimes_n\text{det}(\ca{G}^n)^{(-1)^n}.
\]

We summarize some basic properties of this construction in the following proposition. 
\begin{proposition}
Let $\ca{F}$ be a complex of coherent sheaves which has a bounded locally free resolution. Then
\begin{enumerate} 
\item The construction of $\emph{det}(\ca{F}^\bullet)$ does not depend on the locally free resolution.
\item For every short exact sequence of complexes of sheaves which have bounded locally free resolutions
\[
0\rightarrow \ca{F}^\bullet\xrightarrow{\alpha}\ca{G}^\bullet\xrightarrow{\beta}\ca{H}^\bullet\rightarrow0,
\]
we have a functorial isomorphism
\[
i(\alpha,\beta): \emph{det}(\ca{F}^\bullet)\otimes\emph{det}(\ca{H}^\bullet)\xrightarrow{\sim}\emph{det}(\ca{G}^\bullet).
\]

\item The operator $\emph{det}$ commutes with base change. To be more precise, for every (representable) morphism of Deligne-Mumford stacks
\[
g:\ca{X}\rightarrow \ca{Y},
\]
we have an isomorphism
\[
\emph{det}(\emph{L}g^*)\xrightarrow{\sim}g^*\emph{det}.
\]

\end{enumerate}
\end{proposition}
In the case when $\ca{X}$ is a scheme, the above proposition is proved in \cite{knudsen}. Note that these properties are preserved under flat base change. Therefore they hold for stacks as well.

\subsection{Level structure in quasimap theory}\label{quasimapthy}
In this subsection, we first recall the quasimap theory for nonsingular GIT quotients introduced in \cite{ciocan4}. Then we define the level structure in this setting and discuss its generalizations in orbifold quasimap theory.

Let $Z=\text{Spec}(A)$ be a complex affine algebraic variety in $\bb{C}^n$ and let $G$ be a reductive group acting on it. Let $\theta: G\rightarrow \bb{C}^*$ be a character determining a $G$-equivariant line bundle $L_\theta:=Z\times\bb{C}$. Let $Z^s(\theta)$ and $Z^{ss}(\theta)$ be the stable and semistable loci, respectively. Throughout the paper, we assume $Z^s(\theta)= Z^{ss}(\theta)$ is nonsingular. Furthermore, we assume that $G$ acts freely on $Z^s(\theta)$. It follows that the GIT quotient $Z\sslash_\theta G$ is nonsingular and quasi-projective. For simplicity, we drop $\theta$ from the notation of the GIT quotient. The unstable locus is defined as $Z_{\text{us}}:=Z-Z^s(\theta)$.
 Recall that we can identify the $G$-equivariant Picard group $\text{Pic}^G(Z)$ with the Picard group $\text{Pic}([Z/G])$ of the quotient stack $[Z/G]$ by sending an $G$-equivariant line bundle $L$ to $[L/G]$. Let $\beta\in\text{Hom}_{\bb{Z}}(\text{Pic}^G(Z),\bb{Z})$. 
\begin{definition}[\hspace{-0.0001 cm}\cite{ciocan4}]\label{defofquasimap}
A quasimap is a tuple $(C,p_1,\dots,p_k,P,s)$  where 
\begin{itemize}
\item $(C,p_1,\dots,p_k)$ is a connected, at most nodal, $k$-pointed projective curve of genus $g$,
\item $P$ is a principal $G$-bundle on $C$,
\item $s$ is a section of the induced fiber bundle $P\times_G Z$ on $C$ such that $(P,s)$ is of class $\beta$, i.e., the homomorphism 
\[
\text{Pic}^G(Z)\rightarrow\bb{Z},\quad L\rightarrow \text{deg}_C(s^*(P\times_GL)),
\]
is equal to $\beta$.
\end{itemize} 
We require that there are only finitely many base points, i.e., points $p\in C$ such that $s(p)\in Z_{\text{us}}$. An element $\beta\in\text{Hom}_{\bb{Z}}(\text{Pic}^G(Z),\bb{Z})$ is called $L_\theta$-effective if it can be represented as a finite sum of classes of quasimaps. We also refer to $\beta$ as a curve class. Denote by $E$ the semigroup of $L_\theta$-effective (curve) classes.
\end{definition}
A quasimap $(C,p_1,\dots,p_k,P,s)$ is called \emph{prestable} if the base points are disjoint from the nodes and marked points on $C$. Given a rational number $\epsilon>0$, a prestable quasimap is called $\epsilon$-\emph{stable} if it satisfies the following conditions
\begin{enumerate}
\item $\omega_{C,\text{log}}\otimes\ca{L}_\theta^\epsilon$ is ample, where $\omega_{C,\text{log}}:=\omega_C\big(\sum_{i=1}^k p_i\big)$ is the twisted dualizing sheaf of $C$ and 
\[
\ca{L}_\theta:=u^*(P\times_G L_\theta)\cong P\times_G\bb{C}_\theta.
\]
\item $\epsilon l(x)\leq 1$ for every point $x$ in $C$ where 
\[
l(x):=\text{length}_x\big(\text{coker}(u^*\ca{J})\rightarrow \ca{O}_C\big).
\]
Here $\ca{J}$ is the ideal sheaf of the closed subscheme $P\times_GZ_{\text{us}}$ of $P\times_GZ$.
\end{enumerate}

Let $\ca{Q}^\epsilon_{g,k}(Z\sslash G,\beta)=\{(C,p_1,\dots,p_k,P,s)\}$ be the moduli stack of $\epsilon$-stable quasimaps. It is shown in \cite{ciocan4} that this stack is a separated Deligne-Mumford stack of finite type and it is proper over the affine quotient $Z/_{\text{aff}}\,G:=\text{Spec}(A^G)$. When $Z$ has only local complete intersection singularities, the $\epsilon$-stable quasimap space $\ca{Q}^\epsilon_{g,k}(Z\sslash G,\beta)$ admits a canonical perfect obstruction theory. 
\begin{remark}\label{special}
There are two extreme chambers for the stability parameter $\epsilon$.
\begin{enumerate}
\item $(\epsilon=\infty)$-stable quasimaps. One can check that when $(g,k)\neq(0,0)$ and $\epsilon>1$, the quasimap space $\ca{Q}^\epsilon_{g,k}(Z\sslash G,\beta)$ is isomorphic to the moduli space of stable maps $\overline{\ca{M}}_{g,k}(Z\sslash G,\beta)$. When $(g,k)=(0,0)$, the same holds with $\epsilon>2$. Therefore when $\epsilon$ is sufficiently large, we denote the $\epsilon$-stable quasimap space by 
\[
\ca{Q}^\infty_{g,k}(Z\sslash G,\beta)=\overline{\ca{M}}_{g,k}(Z\sslash G,\beta)
\]
and refer to it as the $(\epsilon=\infty)$-theory.
\item
 $(\epsilon=0+)$-stable quasimaps. Fix $\beta\in E$. For each $\epsilon\in (0,\frac{1}{\beta(L_\theta)}]$, the $\epsilon$-stability is equivalent to the condition that the underlying curve $C$ of a quasimap does not have rational tails and on each rational bridge, the line bundle $\ca{L}_\theta$ has strictly positive degree. Since we need to consider different $\beta$ at the same time, we reformulate the stability condition as
 \[
 \omega_{C,\text{log}}\otimes\ca{L}_\theta^\epsilon\text{ is ample for all }\epsilon\in\bb{Q}_{>0}.
 \] 
Quasimaps which satisfy the above stability condition are referred to as $(\epsilon=0+)$-stable quasimaps.
\end{enumerate}
\end{remark}

To define the level structure, we introduce some notation first. Let $\mathfrak{M}_{g,k}$ be the algebraic stack of pre-stable nodal curves and $\mathfrak{B}un_G$ be the relative moduli stack 
\[
\mathfrak{B}un_G\xrightarrow{\phi}\mathfrak{M}_{g,k}
\] of principal $G$-bundles on the fibers of the universal curve $\mathfrak{C}_{g,k}\rightarrow\mathfrak{M}_{g,k}$. The morphism $\phi$ is smooth.
There is a forgetful morphism which forgets the section $s$ 
\[
\ca{Q}^\epsilon_{g,k}(Z\sslash G,\beta)\xrightarrow{\mu}\mathfrak{B}un_G
\]
 Let $\tilde{\pi}:\mathfrak{C}_{\mathfrak{B}un_{g,k}}\rightarrow\mathfrak{B}un_{g,k}$ be the universal curve which is the pullback of $\mathfrak{C}_{g,k}$ along $\phi$. Let $\tilde{\mathfrak{P}}\rightarrow\mathfrak{C}_{\mathfrak{B}un_{g,k}}$ be the universal principal $G$-bundle. We denote by $\pi:\ca{C}_{g,k}\rightarrow\ca{Q}^\epsilon_{g,k}(Z\sslash G,\beta)$ the universal curve on the quasimap space. Let $\ca{P}\rightarrow\ca{C}_{g,k}$ be the universal principal bundle given by the pullback of $\tilde{\mathfrak{P}}\rightarrow\mathfrak{C}_{\mathfrak{B}un_{g,k}}$.

 \begin{definition}\label{defdet1}
Given a finite-dimensional representation $R$ of $G$, we define the level-$l$ determinant line bundle over $\ca{Q}^\epsilon_{g,k}(Z\sslash G,\beta)$ as 
\begin{equation}\label{eq:maindef}
\ca{D}^{R,l}:=\big(\text{det}\,R\pi_*(\ca{P}\times_G R)\big)^{-l}.
\end{equation}
Alternatively, one can define $\ca{D}^{R,l}$ to be the pullback via $\mu$ of the determinant line bundle 
$\big(\text{det}\,R\pi_*(\tilde{\mathfrak{P}}\times_G R)\big)^{-l}$ on $\mathfrak{B}un_{g,k}$.
\end{definition}

\begin{remark}\label{defdet2}
The definition mentioned in the introduction is the second one. It is conceptually better in the sense that it does not depend on the different moduli spaces over $\mathfrak{B}un_{g,k}$. In our case, these moduli spaces are the $\epsilon$-stable quasimap spaces $\ca{Q}^\epsilon_{g,k}(Z\sslash G,\beta)$ for different $\epsilon$. However, $\mathfrak{B}un_{g,k}$ is an Artin stack, and it is technically more difficult to work with it. Formally, we will use the first definition as the working definition. 
\end{remark}

\begin{remark}\label{notrep}
Note that in Definition \ref{defdet1}, the bundle $\ca{P}\times_G R$ is the pullback of the vector bundle $[Z\times R/G]\rightarrow[Z/G]$ along the evaluation map to the quotient stack $[Z/G]$. Therefore, given a vector bundle $\ca{R}$ on $X$, we can use (\ref{eq:maindef}) to define a determinant line bundle over the moduli space of stable maps $\overline{\ca{M}}_{g,k}(X,\beta)$, even when $X$ is not a GIT quotient. To be more precise, let $\pi:\ca{C}\rightarrow\overline{\ca{M}}_{g,k}(X,\beta)$ be the universal curve and let $\text{ev}:\ca{C}\rightarrow X$ be the universal evaluation map. We define the level-$l$ determinant line bundle as
\[
{\ca{D}}^{R,l}:=\big(\text{det}\,R\pi_*(\text{ev}^*\ca{R})\big)^{-l}.
\]
We will often abuse the notation by referring to the vector bundle $\ca{R}$ as the ``representation'' $R$.
\end{remark}

The above construction can be easily generalized to orbifold quasimap theory. To be more precise, suppose the target can be written as $[Z^s/G]$. Now, we do not assume $G$ acts freely on the stable locus $Z^s(\theta)$. Therefore $[Z^s/G]$ is in general a Deligne-Mumford stack. The quasimap theory for such orbifold GIT targets is established in \cite{ciocan2}. According to \cite[Section 2.4.5]{ciocan2}, we still have universal curves and universal principal $G$-bundles over moduli spaces of $\epsilon$-stable orbifold quasimaps. Therefore the level-$l$ determinant line bundle can still be defined using (\ref{eq:maindef}).

 \subsection{Properties of the level structure in quasimap theory}\label{property}
 
In this subsection, we study the level-$l$ determinant line bundle in the case $\beta=0$ and its pullbacks along some natural morphisms between moduli spaces of quasimaps. An important property of the level structure $\ca{D}^{R,l}$ is that it splits ``correctly'' among nodal strata (see Proposition \ref{cutedge}). In the following discussion, we assume $X=Z\sslash G$ is a GIT quotient so that moduli spaces of quasimaps are defined. When $X$ is a smooth projective variety, but not a GIT quotient, the same results hold for determinant line bundles defined in Remark \ref{notrep}. In fact, the arguments used in the proofs are identical for both cases.
\subsubsection{Mapping to a point}

Assume that $\beta=0$. Then any quasimap is a constant map and the morphism
\[
\ca{Q}^\epsilon_{g,k}(X,0)\xrightarrow{\text{stab$\times$ev}}\overline{M}_{g,k}\times X
\] 
is an isomorphism. Here $\text{stab}: \ca{Q}^\epsilon_{g,k}(X,0)\rightarrow\overline{M}_{g,k}$ denotes the stabilization morphism of source curves of quasimaps, and $\text{ev}:\ca{Q}^\epsilon_{g,k}(X,0)\rightarrow X$ denotes the constant evaluation morphism. Let $P$ be the principal $G$-bundle $Z^s\rightarrow X=Z\sslash G$.
\begin{lemma}\label{trivial}
The universal bundle $\ca{P}$ over the universal curve $\ca{C}=\overline{C}_{g,k}\times X$ is equal to $\pi_2^*(P)$, where $\overline{C}_{g,k}$ is the universal curve over $\overline{M}_{g,k}$ and $\pi_2: \overline{C}_{g,k}\times X\rightarrow X$ is the second projection.
\begin{proof}
In general, there is an evaluation map from the universal curve $\ca{C}$ to the quotient stack $[Z/G]$ and $\ca{P}$ is the pullback of $P$ along this map. The lemma follows from the observation that the evaluation map is given by the second projection $\pi_2$ in this case.
\end{proof}
\end{lemma} 
\begin{corollary}\label{cormappt}
Let $\ca{R}:=P\times_G R$ be the associated vector bundle on $X$ and let $\pi:\overline{C}_{g,k}\rightarrow\overline{M}_{g,k}$ be the canonical morphism. We have
\[ 
\ca{D}^{R,l}=(\wedge^{\emph{rk}(R)g}(R^1\pi_*\ca{O}_{\overline{C}_{g,k}}\boxtimes\ca{R})\big)^l\otimes \big(\wedge^{\emph{rk}(R)}\ca{R}\big)^{-l}.
\]
\end{corollary}
\begin{proof}
 By Lemma \ref{trivial}, we have $\ca{P}\times_G R=\pi_2^*(\ca{R})$. Therefore the pushforward $R\pi_*(\pi_2^*(\ca{R}))$ is equal to $\ca{O}_{\overline{M}_{g,k}}\boxtimes\ca{R}-R^1\pi_*\ca{O}_{\overline{C}_{g,k}}\boxtimes\ca{R}$ via the projection formula. Note that $\text{rk}(R^1\pi_*\ca{O}_{\overline{C}_{g,k}})=g$. 
\end{proof}

\subsubsection{Cutting edges}
For $i=1,2$, we denote by $\text{ev}_{k_i}:\ca{Q}^\epsilon_{g,k_{k_i}}(X,\beta_i)\rightarrow X$ the evaluation morphism at the last marking. Consider the cartesian diagram
\[ \begin{tikzcd}
\ca{Q}^\epsilon_{g,k_1}(X,\beta_1)\times_X\ca{Q}^\epsilon_{g,k_2}(X,\beta_2)\arrow{r}{\Phi} \arrow{d}{} & \ca{Q}^\epsilon_{g,k_1}(X,\beta_1)\times\ca{Q}^\epsilon_{g,k_2}(X,\beta_2)\arrow{d}{\text{ev}_{k_1}\times \text{ev}_{k_2}} \\
X\arrow{r}{\Delta}&X\times X,
\end{tikzcd}
\]
where $\Delta$ is the diagonal embedding of $X$. Let $\ca{C}$ and $\ca{C}'$ denote the universal curves over $\ca{Q}^\epsilon_{g,k_1}(X,\beta_1)\times_X\ca{Q}^\epsilon_{g,k_2}(X,\beta_2)$ and $\ca{Q}^\epsilon_{g,k_1}(X,\beta_1)\times\ca{Q}^\epsilon_{g,k_2}(X,\beta_2)$, respectively. Let $\ca{P}$ and $\ca{P}'$ be the universal principal $G$-bundles over $\ca{C}$ and $\ca{C}'$, respectively. We can define level structures $\ca{D}^{R,l}_{\ca{Q}^\epsilon_{g,k_1}(X,\beta_1)\times\ca{Q}^\epsilon_{g,k_2}(X,\beta_2)}=\ca{D}^{R,l}_{\ca{Q}^\epsilon_{g,k_1}(X,\beta_1)}\boxtimes\ca{D}^{R,l}_{\ca{Q}^\epsilon_{g,k_2}(X,\beta_2)}$ and $\ca{D}^{R,l}_{\ca{Q}^\epsilon_{g,k_1}(X,\beta_1)\times_X\ca{Q}^\epsilon_{g,k_2}(X,\beta_2)}$ using (\ref{eq:maindef}). The following proposition shows that the level structure splits ``correctly'' among nodal strata.

\begin{proposition}\label{cutedge} Let $x:\ca{Q}^\epsilon_{g,k_1}(X,\beta_1)\times_X\ca{Q}^\epsilon_{g,k_2}(X,\beta_2)\rightarrow \mathcal{C}$ be the section corresponding to the node. Then we have
\begin{equation}\label{eq:cuttingedges}
\Phi^*\big(\ca{D}^{R,l}_{\ca{Q}^\epsilon_{g,k_1}(X,\beta_1)}\boxtimes\ca{D}^{R,l}_{\ca{Q}^\epsilon_{g,k_2}(X,\beta_2)}\big)=\ca{D}^{R,l}_{\ca{Q}^\epsilon_{g,k_1}(X,\beta_1)\times_X\ca{Q}^\epsilon_{g,k_2}(X,\beta_2)}\otimes\emph{det}\big(x^*(\ca{P}\times_G R)\big)^{-l}.
\end{equation}
\end{proposition}

\begin{proof}
Consider the following commutative diagram
\[ \begin{tikzcd}
\Phi^*\ca{C}'\arrow{r}{p} \arrow[swap]{dr}{\pi'} &\ca{C}\arrow{d}{\pi} \\
&\ca{Q}^\epsilon_{g,k_1}(X,\beta_1)\times_X\ca{Q}^\epsilon_{g,k_2}(X,\beta_2).
\end{tikzcd}
\]
Notice that $\ca{C}$ is obtained by gluing along two sections of marked points of $\Phi^*\ca{C}'$. 
For any locally free sheaf $\ca{F}$ on $\ca{C}$, we have a short exact sequence (the normalization exact sequence)
\[
0\rightarrow \ca{F}\rightarrow p_*p^* \ca{F}\rightarrow x_*x^* \ca{F}\rightarrow 0.
\]
It induces the following natural isomorphism
\begin{equation}\label{eq:splitting}
\text{det}(R\pi_*(p_*p^* \ca{F}))^{-1}\cong\text{det}(R\pi_*(\ca{F}))^{-1}\otimes\text{det}(R\pi_*(x_*x^* \ca{F}))^{-1}.
\end{equation}
Note that
 \[
\text{det}(R\pi_*(p_*p^* \ca{F}))^{-1}=\text{det}(R\pi'_*(p^* \ca{F}))^{-1}\quad \text{and}\quad
\text{det}(R\pi_*(x_*x^* \ca{F}))^{-1}=\text{det}(x^*\ca{F})^{-1}.
\]
 Now take $\ca{F}=\ca{P}\times_G R$ and $\ca{F}'=\ca{P'}\times_G R$. Finally, the lemma follows from the fact that $p^*\ca{F}\cong\Phi^*\ca{F}'$, equation (\ref{eq:splitting}), and cohomology and base change.
 \end{proof}

\subsubsection{Contractions}\label{Contraction}
Fix $g_1,g_2$ and $k_1,k_2$ such that $g=g_1+g_2$ and $k=k_1+k_2$. We denote the basic gluing maps by
\begin{align*}
r:\overline{M}_{g_1,k_1+1}&\times\overline{M}_{g_2,k_2+1}\rightarrow\overline{M}_{g,k},\\
q:\overline{M}&_{g-1,k+2}\rightarrow\overline{M}_{g,k}.
\end{align*}
Let $k'$ be a non-negative integer. Let $\underline{k}'=(k_1',\dots,k'_{m+1})$ and $\underline{\beta}=(\beta_1,\dots,\beta_{m+1})$ be partitions of $k'$ and $\beta$, respectively. For simplicity, we denote by $\ca{Q}^\epsilon_{m,\underline{k}',\underline{\beta}}$ the fiber product
\[
\ca{Q}^\epsilon_{g_1,k_1+k'_1+1}(X,\beta_1)\times_X\underbrace{\ca{Q}^\epsilon_{0,2+k'_3}(X,\beta_3)\times_X\dots\ca{Q}^\epsilon_{0,2+k'_{m+1}}(X,\beta_{m+1})}_{m-1\ \text{factors}}\times_X\ca{Q}^\epsilon_{g_2,k_2+k'_2+1}(X,\beta_2)
\]
Let $\text{st}:\ca{Q}^\epsilon_{g,k+k'}(X,\beta)\rightarrow\overline{M}_{g,k}$ be the morphism defined by forgetting the (quasi)map and the last $k'$ markings, then stabilizing the source curve. 
Consider the following commutative diagram.
\[ \begin{tikzcd}
\bigsqcup\ca{Q}^\epsilon_{m,\underline{k}',\underline{\beta}}\arrow{r}{\sqcup\, r_{m,\underline{k}',\underline{\beta}}} \arrow{d}{} & \ca{Q}^\epsilon_{g,k+k'}(X,\beta)\arrow{d}{\text{st}} \\
\overline{M}_{g_1,k_1+1}\times\overline{M}_{g_2,k_2+1}\arrow{r}{r}& \overline{M}_{g,k}
\end{tikzcd}
\]
Here the disjoint union is over partitions $\underline{k}'$ of $k'$ and partitions $\underline{\beta}$ of $\beta$. The above commutative diagram induces a morphism
\[
\Psi_m:\bigsqcup\ca{Q}^\epsilon_{m,\underline{k}',\underline{\beta}}\rightarrow\big(\overline{M}_{g_1,k_1+1}\times\overline{M}_{g_2,k_2+1}\big)\times_{\overline{M}_{g,k}} \ca{Q}^\epsilon_{g,k+k'}(X,\beta).
\]
Using the same argument as in the proof of \cite[Proposition 11]{lee1}, one can show that the virtual structure sheaves satisfy
\[
\sum_m(-1)^{m+1}\Psi_{m*}\sum_{\underline{k}',\underline{\beta}}\ca{O}^\vir_{\ca{Q}^\epsilon_{m,\underline{k}',\underline{\beta}}}=r^!\,\ca{O}^\vir_{\ca{Q}^\epsilon_{g,k+k'}(X,\beta)}.
\]

Let $\ca{C}_{g,k+k'}$ and $\ca{C}_{m,\underline{k}',\underline{\beta}}$ be the universal curves on $\ca{Q}^\epsilon_{g,k+k'}(X,\beta)$ and $\ca{Q}^\epsilon_{m,\underline{k}',\underline{\beta}}$, respectively. Let $\ca{P}_{g,k+k'}$ and $\ca{P}_{m,\underline{k}',\underline{\beta}}$ be the corresponding universal principal $G$-bundles. The level structures $\ca{D}^{R,l}_{\ca{Q}^\epsilon_{g,k+k'}(X,\beta)}$ and $\ca{D}^{R,l}_{\ca{Q}^\epsilon_{m,\underline{k}',\underline{\beta}}}$ can be defined using (\ref{eq:maindef}). To prove that quantum $K$-theory with level structure satisfies the same axioms as Givental-Lee's quantum $K$-theory, we need the following proposition.

\begin{proposition}\label{cont1}
We have
\[
\ca{D}^{R,l}_{\ca{Q}^\epsilon_{m,\underline{k}',\underline{\beta}}}=(r_{m,\underline{k}',\underline{\beta}})^*\,\ca{D}^{R,l}_{\ca{Q}^\epsilon_{g,k+k'}(X,\beta)}.
\]
\end{proposition}
\begin{proof}
The proposition follows from the following cartesian diagram and cohomology and base change.
\[ \begin{tikzcd}
\ca{C}_{m,\underline{k}',\underline{\beta}}\arrow{r}\arrow{d}&\ca{C}_{g,k+k'}\arrow{d}\\
\ca{Q}^\epsilon_{m,\underline{k}',\underline{\beta}}\arrow{r}{r_{m,\underline{k}',\underline{\beta}}} & \ca{Q}^\epsilon_{g,k+k'}(X,\beta)
\end{tikzcd}
\]
\end{proof}

\section{The $K$-theoretic quasimap theory with level structure} \label{quantumk}
In this section, we first define the $K$-theoretic quasimap invariants with level structure and their permutation-equivariant version. For most of the discussion, we assume that the GIT quotient $Z\sslash G$ is nonsingular and projective. The cases when the target $[Z^s/G]$ is noncompact or an orbifold are discussed at the end of this section.

\subsection{$K$-theoretic quasimap invariants with level structure}
In this subsection, we first briefly recall Givental-Lee's quantum $K$-theory. Then we define $K$-theoretic quasimap invariants with level structure. 

The quantum $K$-theory or $K$-theoretic Gromov-Witten theory was introduced by Givental-Lee \cite{givental3,lee1}. Let $X$ be a smooth projective variety and let $\overline{\ca{M}}_{g,k}(X,\beta)$ be the moduli space of stable maps to $X$. The moduli space is known to be a proper Deligne-Mumford stack (see for example \cite{behrend2}). In particular, for any coherent sheaf $\mathcal{E}$ on $\overline{\ca{M}}_{g,k}(X,\beta)$, we can consider its $K$-theoretic pushforward to the point $\text{Spec}\,\bb{C}$, i.e., we can take its Euler characteristic
\[
\chi(\mathcal{E})=\sum_i(-1)^ih^i(\mathcal{E}),
\]
where $h^i(\ca{E}):=\text{dim}_{\bb{C}}\,H^i(\overline{\ca{M}}_{g,k}(X,\beta),\ca{E})$.

From the perfect obstruction theory, Lee \cite{lee1} constructs a virtual structure sheaf $\ca{O}^{\text{vir}}\in K_0(\overline{\ca{M}}_{g,k}(X,\beta))$, where $K_0(\overline{\ca{M}}_{g,k}(X,\beta))$ denotes the Grothendieck group of coherent sheaves on $\overline{\ca{M}}_{g,k}(X,\beta)$. The virtual structure sheaf $\ca{O}^{\text{vir}}$ has the following properties:
\begin{enumerate}
\item If the obstruction sheaf is trivial and hence $\overline{\ca{M}}_{g,k}(X,\beta)$ is smooth, then $\ca{O}^{\text{vir}}$ is the structure sheaf of $\overline{\ca{M}}_{g,k}(X,\beta)$.

\item If the obstruction sheaf $\text{Obs}$ is locally free, then $\ca{O}^{\text{vir}}=\sum_i(-1)^i\wedge^i\text{Obs}^\vee$. Here $\wedge^i\, \text{Obs}^\vee$ denotes the $i$-th wedge product of the dual of the obstruction bundle.
\end{enumerate}

Since we assume $X$ to be smooth, the Grothendieck group of locally free sheaves on $X$, denoted by $K^0(X)$, is isomorphic to the Grothendieck group of coherent sheaves $K_0(X)$. We denote both of them by $K(X)$. Suppose that $E_i$ are $K$-theory elements in $K(X)$ and let $L_i$ denote the $i$-th cotangent line bundles. The \emph{$K$-theoretic Gromov-Witten invariants} are defined by
\[
\big\langle E_1L_1^{l_1},\dots, E_kL_k^{l_k}\big\rangle_{g,k,\beta}=\chi\big(\overline{\ca{M}}_{g,k}(X,\beta),\prod_i\text{ev}_i^* E_i\otimes L_i^{l_i}\otimes\ca{O}^{\text{vir}}\big),
\]
where $\text{ev}_i:\overline{\ca{M}}_{g,k}(X,\beta)\rightarrow X$ are the evaluation morphisms at the $i$-th marking. Let $E\subset H_2(X,\bb{Z})$ be the semigroup generated by effective curve classes on $X$. We define the \emph{quantum $K$-potential of genus 0} by
\[
\ca{F}(t,Q):=\frac{1}{2}(t,t)+\sum_{k=0}^\infty\sum_{\beta\in E}\frac{Q^\beta}{k!}\langle t,\dots, t\rangle_{0,k,\beta},
\]
where $t\in K(X)_{\bb{Q}}:=K(X)\otimes{\bb{Q}}$ and $(t,t):=\chi(t\otimes t)$ is the Mukai pairing. 
Let $\phi_0=\ca{O}_X,\phi_1,\phi_2\dots$ be a basis of $K(X)_{\bb{Q}}$. One can define the ``quantized'' pairing on $K(X)_{\bb{Q}}$ by
\[
((\phi_i, \phi_j)):=\ca{F}_{ij}=\partial_{t_i}\partial_{t_j}\ca{F}(t,Q).
\]
It is showed in \cite{lee1} that quantum $K$-theory satisfies all the usual axioms of cohomological Gromov-Witten theory except the flat identity axiom. We refer the reader to \cite{lee1} for more details. 

In the following discussion, we assume $X$ can be represented as a GIT quotient $Z\sslash G$. As mentioned before, the moduli space of stable maps $\overline{\ca{M}}_{g,k}(Z\sslash G,\beta)$ can be identified with $\ca{Q}^\epsilon_{g,k}(Z\sslash G,\beta)$ for large $\epsilon$. According to \cite{ciocan4},  for general $\epsilon$, the $\epsilon$-stable quasimap space $\ca{Q}^\epsilon_{g,k}(Z\sslash G,\beta)$ is proper and it admits a two-term perfect obstruction theory, assuming $Z$ has only lci singularities. Hence by the result in \cite{lee1}, one can construct a virtual structure sheaf $\ca{O}^{\text{vir}}$ on $\ca{Q}^\epsilon_{g,k}(Z\sslash G,\beta)$.
\begin{definition}
The \emph{$K$-theoretic quasimap invariants of level $l$} are defined by
\[
\big\langle E_1L_1^{l_1},\dots, E_kL_k^{l_k}\big\rangle_{g,k,\beta}^{Z\sslash G,R,l,\epsilon}=\chi\big(\ca{Q}^\epsilon_{g,k}(Z\sslash G,\beta),\prod_i\text{ev}_i^* E_i\otimes L_i^{l_i}\otimes\ca{O}^{\text{vir}}\otimes \ca{D}^{R,l}\big)\in\bb{Z},
\]where $E_i\in K(Z\sslash G)\otimes\bb{Q}$. 
\end{definition}
We shall usually suppress $Z\sslash G$ from the notation if there is no confusion. Note that these invariants are all integers.

\subsection{Quasimap graph space and $\ca{J}^{R,l,\epsilon}$-function}\label{quasigraph}
In this subsection, we first recall the definition and properties of the $\epsilon$-stable quasimap graph space. Then we define an important generating series $\ca{J}^{R,l,\epsilon}$ of $K$-theoretic quasimap invariants of level $l$. 

Given a rational number $\epsilon>0$, the \emph{quasimap graph space}, denoted by $\ca{QG}_{g,k}^\epsilon(Z\sslash G,\beta)$, is introduced in \cite{ciocan4}. It is the moduli space of the tuples
\[
((C,x_1,\dots x_k),P,u,\varphi),
\]
where $((C,x_1,\dots x_k),P,u)$ is a prestable quasimap, satisfying $\epsilon l(x)<1$ for every point $x$ on $C$, and the new data $\varphi$ is a degree 1 morphism from $C$ to $\bb{P}^1$. The curve $C$ has a unique rational component $C_0$ such that $\varphi|_{C_0}:C_0\rightarrow\bb{P}^1$ is an isomorphism and the complement $C/ C_0$ is contracted by $\varphi$. The ampleness condition imposed on the tuples is modified to:
\[
\omega_{\overline{C\setminus C_0}}(\sum x_i+\sum y_j)\otimes \ca{L}_\theta^\epsilon \text{ is ample}, 
\]
where $x_i$ are marked points on $\overline{C\setminus C_0}$ and $y_i$ are the nodes $\overline{C\setminus C_0}\cap C_0$. It is shown in \cite{ciocan4} that the quasimap graph space is also a separated Deligne-Mumford stack which is proper over the affine quotient. Moreover, when $Z$ has only lci singularities, the canonical obstruction theory on the graph space is perfect. Similarly, we can define the level-$l$ determinant line bundle $\ca{D}_{\ca{QG}}^{R,l}$ on $\ca{QG}_{g,k}^\epsilon(Z\sslash G,\beta)$ using the universal principal $G$-bundles over its universal curve.

There is a natural $\bb{C}^*$-action on the graph spaces. Let $[x_0,x_1]$ be homogeneous coordinates on $\bb{P}^1$, and set $0:=[1,0]$ and $\infty:=[0,1]$. We consider the standard $\bb{C}^*$-action on $\bb{P}^1$:
\begin{equation}\label{eq:actionst}
t\cdot[x_0,x_1]=[tx_0,x_1],\quad\forall t\in\bb{C}^*.
\end{equation}
It induces an action on the $\epsilon$-stable quasimap graph space $\ca{QG}_{g,k}^\epsilon(Z\sslash G,\beta)$ by rescaling the parametrized rational component. According to \cite[\textsection 4.1]{ciocan1}, the $\bb{C}^*$-fixed locus can be described as
\[
(\ca{QG}_{g,k}^\epsilon(Z\sslash G,\beta))^{\bb{C}^*}=\coprod F^{g_1,k_1,\beta_1}_{g_2,k_2,\beta_2},
\]
where the disjoint union is over all possible splittings
\[
g=g_1+g_2,\quad k=k_1+k_2,\quad \beta=\beta_1+\beta_2,
\]
with $g_i,k_i\geq 0$ and $\beta_i$ effective. In the stable cases, an $\epsilon$-stable parametrized quasimap $((C,x_1,\dots,x_k),P,u,\varphi)\in  F^{g_1,k_1,\beta_1}_{g_2,k_2,\beta_2}$ is obtained by gluing two $\epsilon$-stable quasimaps of types $(g_1,k_1,\beta_1)$ and $(g_2,k_2,\beta_2)$ to a constant map $\bb{P}^1\rightarrow p\in Z\sslash G$ at $0$ and $\infty$, respectively. Therefore, the component $ F^{g_1,k_1,\beta_1}_{g_2,k_2,\beta_2}$ is isomorphic to the fiber product
\[
\ca{Q}_{g_1,k_1+\bullet}^\epsilon(Z\sslash G,\beta_1)\times_{Z\sslash G}\ca{Q}_{g_2,k_2+\bullet}^\epsilon(Z\sslash G,\beta_2)
\]
over the evaluation maps at the special marked points $\bullet$. When one of the components at $0$ or $\infty$ is unstable, we use the following conventions.
\begin{enumerate}
\item For the unstable cases $(g_1,k_1,\beta_1)=(0,0,0)$ or $(0,1,0)$ (and likewise for $(g_2,k_2,\beta_2)$), we define
\[
\ca{Q}_{0,0+\bullet}^\epsilon(Z\sslash G,0):=Z\sslash G,\quad\ca{Q}_{0,1+\bullet}^\epsilon(Z\sslash G,0):=Z\sslash G,\quad\text{ev}_\bullet =\text{Id}_{Z\sslash G}.
\]
\item For the unstable cases $(g_1,k_1,\beta_1)=(0,0,\beta_1)$ with $\beta_1\neq0$ and $\epsilon\leq\frac{1}{\beta_1(L_\theta)}$, we denote by
\[
\ca{Q}_{0,0+\bullet}(Z\sslash G,\beta_1)_0
\]
the moduli space of quasimaps $(C=\bb{P}^1,P,u)$ such that $u(x)\in P\times_GZ^s$ for $x\neq0\in\bb{P}^1$ and $0\in\bb{P}^1$ is a base point of length $\beta_1(L_\theta)$. Similarly, we define $
\ca{Q}_{0,0+\bullet}(Z\sslash G,\beta_2)_\infty$ to be the moduli space of quasimaps whose only base point is of length $\beta_2(L_\theta)$ and located at $\infty$. Using these definitions we have
\[
F^{g,k,\beta_1}_{0,0,\beta_2}\cong\ca{Q}_{g,k+\bullet}^\epsilon(Z\sslash G,\beta_1)\times_{Z\sslash G}\ca{Q}_{0,0+\bullet}(Z\sslash G,\beta_2)_\infty
\]
for $k\geq 1$ and $\epsilon\leq \frac{1}{\beta_2(L_\theta)}$. Similarly we have 
\[
F_{g,k,\beta_2}^{0,0,\beta_1}\cong\ca{Q}_{0,0+\bullet}(Z\sslash G,\beta_1)_0\times_{Z\sslash G}\ca{Q}_{g,k+\bullet}^\epsilon(Z\sslash G,\beta_2)
\]
for $k\geq 1$ and $\epsilon\leq \frac{1}{\beta_1(L_\theta)}$. 
When $g=k=0$ and $\epsilon\leq\text{min}\{ \frac{1}{\beta_1(L_\theta)}, \frac{1}{\beta_2(L_\theta)}\}$, we have
\[
F_{0,0,\beta_2}^{0,0,\beta_1}\cong\ca{Q}_{0,0+\bullet}(Z\sslash G,\beta_1)_0\times_{Z\sslash G}\ca{Q}_{0,0+\bullet}(Z\sslash G,\beta_2)_\infty.
\]
\end{enumerate}

We denote by $\ca{R}$ the vector bundle $Z\times_G R\rightarrow [Z/G]$ and its restriction to $Z\sslash G$. We define the twisted pairing on $K(Z\sslash G)_\bb{Q}$ by
\begin{equation}\label{eq:twistedpairing}
(u,v)^{R,l}:=\chi\big(u\otimes v\otimes(\text{det}\,\ca{R})^{-l}\big),\quad\text{where}\ u,v\in K(Z\sslash G)_\bb{Q}.
\end{equation}
Let $\{\phi_a\}$ be a basis of $K(Z\sslash G)_\bb{Q}$ and let $\{\phi^a\}$ be the dual basis with respect to the above twisted pairing $(\cdot,\cdot)^{R,l}$. Let $t=\sum_it^i\phi_i\in K(Z\sslash G)_\bb{Q}$. We define the \emph{$\ca{J}^{R,l,\epsilon}$-function of level $l$} to be
\begin{equation}\label{eq:jfun}
\ca{J}^{R,l,\epsilon}(t,Q)=1-q+t+\sum_a\sum_{(k,\beta)\neq(0,0),(1,0)}\frac{Q^\beta}{k!}\phi^a\bigg\langle\frac{\phi_a}{1-qL},t,\dots,t\bigg\rangle_{0,k+1,\beta}^{R,l,\epsilon}.\footnote{In Ciocan-Fontanine-Kim's convention, the $J$-function starts at 1. The definition given here agrees with Givental's convention in which the $J$-function starts at the dilaton shift $1-q$. }
\end{equation}
In the above summation, the quasimap moduli spaces are empty when $k=0,\beta\neq0,\beta(L_\theta)\leq1/\epsilon$, and the unstable terms are defined by $\bb{C}^*$-localization on the graph space $\ca{QG}^\epsilon_{0,0}(Z\sslash G,\beta)$. To be more precise, we consider the fixed point locus $F_{0,\beta}:=\ca{Q}_{0,0+\bullet}(Z\sslash G,\beta)_0$ of the $\bb{C}^*$-action. The unstable terms in (\ref{eq:jfun}) are defined to be
\[
(1-q)\sum_a\sum_{\beta\neq0,\,\beta(L_\theta)\leq 1/\epsilon}Q^\beta\chi\bigg(F_{0,\beta},\ca{O}_{F_{0,\beta}}^{\text{vir}}\otimes \text{ev}^*(\phi_a)\otimes\bigg(\frac{\text{tr}_{\bb{C}^*}\ca{D}^{R,l}}{\text{tr}_{\bb{C}^*}\wedge^* N_{F_{0,\beta}}^\vee}\bigg)\bigg)\phi^a.
\]
where $N_{F_{0,\beta}}$ is the \emph{virtual} normal bundle of the fixed locus $F_{0,\beta}$ in $\ca{QG}^\epsilon_{0,0}(Z\sslash G,\beta)$ and $\wedge^* N_{F_{0,\beta}}^\vee:=\sum_i(-1)^i\wedge^iN_{F_0}^\vee$ is the \emph{$K$-theoretic Euler class} of $N_{F_{0,\beta}}$. Here the trace of a $\bb{C}^*$-equivariant bundle $V$, when restricted to the fixed point locus, is a virtual bundle defined by the eigenspace decomposition with respect to the $\bb{C}^*$-action, i.e., we have 
\[
\text{tr}_{\bb{C}^*}(V):=\sum_iq^i\,V(i),
\]
where $t\in\bb{C}^*$ acts on $V(i)$ as multiplication by $t^i$.

For $1<\epsilon\leq\infty$, i.e., the $(\epsilon=\infty)$-theory, $(\ref{eq:jfun})$ defines $\ca{J}$-function in the quantum $K$-theory of level $l$. In this case, we use $\langle\cdot\rangle^{R,l,\infty}$ or simply $\langle\cdot\rangle^{R,l}$ to denote quantum $K$-invariants of level $l$. Following Givental-Tonita \cite{givental2}, we introduce the \emph{symplectic loop space formalism}.
Recall that $E$ denotes the semigroup of $L_\theta$-effective curve classes on $Z\sslash_\theta G$. The \emph{Novikov ring} $\bb{C}[[Q]]$ is defined as
 \[
 \bb{C}[[Q]]:=\overline{\{\sum_{\beta\in E}c_{\beta}Q^\beta|c_\beta\in\bb{C}\} }.
 \]
Here the completion is taken with respect to the $\mathfrak{m}$-adic topology, where $\mathfrak{m}$ denotes the maximal ideal generated by nonzero elements of $E$.
We define the \emph{loop space} as
\[
\ca{K}:=[K(Z\sslash G)\otimes\bb{C}(q)]\otimes\bb{C}[[Q]],
\]
where $\bb{C}(q)$ is the field of complex rational functions in $q$. By viewing the elements in $\bb{C}(q)\otimes\bb{C}[[Q]]$ as the coefficients, we extend the twisted pairing $(\cdot,\cdot)^{R,l}$ to $\ca{K}$ via linearity. There is a natural symplectic form $\Omega$ on $\ca{K}$ defined by 
 \begin{equation}\label{eq:sympform}
\Omega(f,g):=[\text{Res}_{q=0}+\text{Res}_{q=\infty}](f(q),g(q^{-1}))^{R,l}\frac{dq}{q},\quad\text{where}\ f,q\in\ca{K}.
\end{equation}
With respect to $\Omega$, there is a \emph{Lagrangian polarization} $\ca{K}=\ca{K}_+\oplus\ca{K}_-$, where
\[
\ca{K}_+=[K(Z\sslash G)\otimes\bb{C}[q,q^{-1}]]\otimes\bb{C}[[Q]]\quad\text{and}\quad \ca{K}_-=\{f\in\ca{K}|f(0)\neq\infty,f(\infty)=0\}.
\]
As before, let $\{\phi_a\}$ be a basis of $K(Z\sslash G)_\bb{Q}$ and let $\{\phi^a\}$ be the dual basis with respect to the twisted pairing $(\cdot,\cdot)^{R,l}$. Let $\mb{t}(q)=\sum_{i,j}t^i_j\phi_iq^j\in\ca{K}_+$ be an arbitrary Laurent polynomial. We define the \emph{big $\ca{J}$-function of level $l$} to be the function $\ca{J}^{R,l}(\mb{t}(q),Q):\ca{K}_+\rightarrow\ca{K}$ given by
\[
\ca{J}^{R,l}(\mb{t}(q),Q)=1-q+\mb{t}(q)+\sum_a\sum_{(k,\beta)\neq(0,0),(1,0)}\frac{Q^\beta}{k!}\phi^a\bigg\langle\frac{\phi_a}{1-qL},\mb{t}(L),\dots,\mb{t}(L)\bigg\rangle_{0,k+1,\beta}^{R,l,\infty}.
\]
Define the \emph{genus-0 $K$-theoretic descendant potential of level $l$} by
\begin{equation}\label{eq:potential}
\ca{F}^{R,l}(\mb{t},Q):=\sum_{k,\beta}\frac{Q^\beta}{k!}\langle \mb{t}(L),\dots, \mb{t}(L)\rangle_{0,k,\beta}^{R,l,\infty}.
\end{equation}
We can identify the cotangent bundle $T^*\ca{K}_+$ with the symplectic loop space $\ca{K}$ via the Lagrangian polarization and the \emph{dilaton shift} $f\rightarrow f+(1-q)$. Then $\ca{J}^{R,l}$ coincides with the differential of the descendant potential up to the dilaton shift, i.e., we have
\[
\ca{J}^{R,l}=1-q+\mb{t}(q)+d_t\ca{F}^{R,l}(\mb{t},Q).
\]
In the case $l=0$, the above fact is proved in \cite[\textsection{2}]{givental2}. The same argument works for arbitrary $l$. 

For $(\epsilon=0+)$-stable quasimap theory, the definition (\ref{eq:jfun}) gives the $\ca{I}$-function of level $l$ of $Z\sslash G$:
\begin{align*}
\ca{I}^{R,l}(t,Q)&:=\ca{J}^{R,l,0+}(t,Q)/(1-q)\\
&=1+\frac{t}{1-q}+\sum_a\sum_{\beta\neq0}Q^\beta\chi\bigg(F_{0,\beta},\ca{O}_{F_{0,\beta}}^{\text{vir}}\otimes \text{ev}^*(\phi_a)\otimes\bigg(\frac{\text{tr}_{\bb{C}^*}\ca{D}^{R,l}}{\text{tr}_{\bb{C}^*}\wedge^* N_{F_{0,\beta}}^\vee}\bigg)\bigg)\phi^a
\\
&+\sum_a\sum_{k\geq 1,(k,\beta)\neq(1,0)}\frac{Q^\beta}{k!}\phi^a\bigg\langle\frac{\phi_a}{(1-q)(1-qL)},t,\dots,t\bigg\rangle_{0,k+1,\beta}^{R,l,\epsilon=0+},
\end{align*}
where $t\in K(Z\sslash G)_\bb{Q}$.

\subsection{The permutation-equivariant quasimap $K$-theory with level structure}\label{perm}
Givental \cite{givental11,givental12,givental13,givental14,givental15,givental16,givental17,givental18,givental19,givental21,givental22} introduced the \emph{permutation-equivariant} quantum $K$-theory, which takes into account the $S_n$-action on the moduli spaces of stable maps by permuting the marked points. The definition can be easily generalized to incorporate the level structure.

Let $\Lambda$ be a $\lambda$-algebra, i.e. an algebra over $\bb{Q}$ equipped with abstract Adams operations $\Psi^k,k=1,2,\dots$. Here $\Psi^k:\Lambda\rightarrow\Lambda$ are ring homomorphisms which satisfy $\Psi^r\Psi^s=\Psi^{rs}$ and $\Psi^1=\text{id}$. We often assume that $\Lambda$ includes the Novikov variables, the algebra of symmetric polynomials in a given number of variables, and the torus equivariant $K$-ring of a point. We also assume that $\Lambda$ has a maximal ideal $\Lambda_+$ and is equipped with the $\Lambda_+$-adic topology. For example, we can choose 
\[
\Lambda=\bb{Q}[[N_1,N_2,\dots]][[Q]][\Lambda_0^\pm,\dots,\Lambda_N^\pm],
\]
where $N_i$ are the Newton polynomials (in infinitely or finitely many variables) and $Q$ denotes the Novikov variable(s). The parameters $\Lambda_i$ denote the torus-equivariant parameters. The Adams operations $\Psi^r$ act on $N_m$ and $Q$ by $\Psi^r(N_m)=N_{rm}$ and $\Psi^r(Q^\beta)=Q^{r\beta}$, respectively. We assume their actions on the torus-equivariant parameters are trivial.

Similar to the ``ordinary" quasimap $K$-theory with level structure, we define the loop space by
\[
\ca{K}:=[K(Z\sslash G)\otimes\Lambda]\otimes\bb{C}(q).
\]
As before, it is equipped with a symplectic form defined by (\ref{eq:sympform}), and it has a Lagrangian polarization \[
\ca{K}=\ca{K}_+\oplus\ca{K}_-,
\]
where $\ca{K}_+$ is the subspace of Laurent polynomials in $q$ and $\ca{K}_-$ is the subspace of reduced rational functions which are regular at $q=0$ and vanish at $q=\infty$.

Consider the natural $S_k$-action on the quasimap moduli space $\ca{Q}^\epsilon_{g,k}(Z\sslash G,\beta)$ by permuting the $k$ marked points. Notice that the virtual structure sheaf $\ca{O}_{\ca{Q}^\epsilon_{g,k}(Z\sslash G,\beta)}$ and the determinant line bundle $\ca{D}^{R,l}$ are invariant under this action. Therefore we have the following $S_k$-module
\[
\big[\mb{t}(L),\dots,\mb{t}(L)\big]_{g,k,\beta}:=\sum_m(-1)^mH^m\big(\ca{O}_{\ca{Q}^\epsilon_{g,k}(Z\sslash G,\beta)}^{\text{vir}}\otimes \ca{D}^{R,l}\otimes_{i=1}^k\mb{t}(L_i)\big),
\]
for $\mb{t}(q)\in\ca{K}_+$.
\begin{definition}\label{definquasimap}
The correlators of the permutation-equivariant quasimap $K$-theory of level $l$ are defined by
\[
\big\langle \mb{t}(L),\dots,\mb{t}(L)\big\rangle_{g,k,\beta}^{R,l,\epsilon,S_k}:=\pi_*\big(\ca{O}_{\ca{Q}^\epsilon_{g,k}(Z\sslash G,\beta)}^{\text{vir}}\otimes \ca{D}^{R,l}\otimes_{i=1}^k \mb{t}(L_i)\big),
\]	
where $\pi_*$ is the $K$-theoretic pushforward along the projection
\[
\pi:\big[\ca{Q}^\epsilon_{g,k}(Z\sslash G,\beta)/S_k\big]\rightarrow[\text{pt}].
\]
\end{definition}
\begin{remark}\label{symmetrized}
When the $\lambda$-algebra $\Lambda$ is chosen to be $\bb{Z}[[Q]]$, we refer to the invariants as the \emph{symmetrized} invariants. The pushforward map in Definition \ref{definquasimap} carries information only about the dimensions of the $S_k$-invariant parts of sheaf cohomology $\big[\mb{t}(L),\dots,\mb{t}(L)\big]_{g,k,\beta}$. We refer to \cite[Example 4]{givental11} for more details.
\end{remark}
For the permutation-equivariant quasimap $K$-theory, we also consider the $J^\epsilon$-function, and define the cone $\ca{L}_{S_\infty}$ to be the range of the $J^\infty$-function.
\begin{definition}
The permutation-equivariant $K$-theoretic $J^\epsilon$-function of $Z\sslash G$ of level $l$ is defined by
\begin{equation}\label{eq:permjfun}
\ca{J}^{R,l,\epsilon}_{S_\infty}(\mb{t}(q),Q):=1-q+\mb{t}(q)+\sum_a\sum_{(k,\beta)\neq(0,0),(1,0)}Q^\beta\bigg\langle\frac{\phi_a}{1-qL},\mb{t}(L),\dots,\mb{t}(L)\bigg\rangle_{0,k+1,\beta}^{R,l,\epsilon,S_k}\ \phi^a,
\end{equation}
 where the unstable terms in the summation are the same as those in (\ref{eq:jfun}).
\end{definition}
Note that in the above definition of permutation-equivariant $J$-function, we do not need to divide each term by $k!$.

\begin{definition}
We define the Givental's cone $\ca{L}^{R,l}_{S_\infty}$ as the range of $\ca{J}_{S_\infty}^{R,l,\infty}$, i.e., 
\[
\ca{L}^{R,l}_{S_\infty}:=\bigcup_{\mb{t}(q)\in\ca{K}_+}\ca{J}_{S_\infty}^{R,l,\infty}(\mb{t}(q),Q)\subset\ca{K}.
\]
\end{definition}
\begin{remark}
In the ordinary, i.e. permutation-\emph{non}-equivariant, quantum $K$-theory, the range of the $\ca{J}$-function is a cone which coincides with the differential of the descendant potential (up to the dilaton shift). Therefore the range of the ordinary $K$-theoretic $\ca{J}$-function is a Lagrangian cone in the loop space $\ca{K}$. However, in the permutation-equivariant theory, it is explained in \cite{givental17} that the cone $\ca{L}_{S_\infty}^{R,l}$ is not Lagrangian.
\end{remark}

\subsection{The level structure in equivariant quasimap theory and orbifold quasimap theory}\label{orbthy}
When $Z\sslash G$ is not proper, one can still define the \emph{equivariant} quasimap invariants if $Z\sslash G$ has an additional torus action such that the the fixed point loci in the quasimap moduli spaces are proper. It is explained in \cite[\textsection 6.3] {ciocan4} how to define cohomological quasimap invariants via virtual localization. Similarly, one can define equivariant $K$-theoretic quasimap invariants (with level structure) for noncompact GIT targets using the $K$-theoretic virtual localization formula (see \cite[\textsection{3.2}]{qu1}). With this understood, we define the $\ca{J}^{R,l,\epsilon}$-function of equivariant $K$-theoretic quasimap invariants of level $l$ using (\ref{eq:jfun}). Its permutation-equivariant generalization is straightforward.

If we do not assume $G$ acts freely on the stable locus $Z^s$, then the target $X:=[Z^s/G]$ is naturally an orbifold. For such orbifold GIT targets, a quasimap is a tuple $((C,x_1,\dots,x_k),[u])$ where $(C,x_1,\dots,x_k)$ is a $k$-pointed, genus $g$ twisted curve (see \cite[\textsection 4]{abramovich2}) and $[u]$ is a representable morphism from $(C,x_1,\dots,x_k)$ to $X$. We refer the reader to \cite[\textsection 2.3]{ciocan2} for the details of the $\epsilon$-stability imposed on those tuples. We denote by $\ca{Q}^\epsilon_{g,k}(X,\beta)$ the moduli stack of $\epsilon$-stable quasimaps to the orbifold $X$. It is shown in \cite[Thm. 2.7]{ciocan2} that this moduli stack is Deligne-Mumford and proper over the affine quotient. Furthermore if $Z$ only has l.c.i. singularties, then $\ca{Q}^\epsilon_{g,k}(X,\beta)$ has a canonical perfect obstruction theory. Let $\ca{C}\rightarrow \ca{Q}^\epsilon_{g,k}(X,\beta)$ be the universal curve. The universal principal $G$-bundle $\ca{P}\rightarrow\ca{C}$ is defined as the pullback of the principal $G$-bundle $Z\rightarrow[Z/G]$ via the universal morphism $[u]:\ca{C}\rightarrow[Z/G]$. In the orbifold setting, we can still define the level-$l$ determinant line bundle $\ca{D}^{R,l}$ using (\ref{eq:maindef}). 

According to \cite[\textsection 2.5.1]{ciocan2}, there are natural evaluation morphisms
\[
\text{ev}_i:\ca{Q}^\epsilon_{g,k}(X,\beta)\rightarrow \bar{I}_\mu X,\ \big((C,x_1,\dots,x_k),[u]\big)\mapsto[u]_{|_{x_i}},\quad\text{for}\ i=1,\dots,k.
\]
Here $\bar{I}_\mu X$ denotes the rigidified cyclotomic inertia stack of $X$ which parameterizes representable maps from gerbes banded by finite cyclic groups to $X$. Let $L_i$ be the universal cotangent line bundle whose fiber at $((C,x_1\dots,x_k),[u])$ is the cotangent space of the coarse curve $\underline{C}$ of $C$ at the $i$-th marked point $\underline{x_i}$. For non-negative integers $l_i$ and classes $E_i\in K^0(\bar{I}_\mu X)\otimes\bb{Q}$, we define the $K$-theoretic quasimap invariants of level $l$ as
\[
\big\langle E_1L_1^{l_1},\dots, E_kL_k^{l_k}\big\rangle_{g,k,\beta}^{X,R,l,\epsilon}=\chi\big(\ca{Q}^\epsilon_{g,k}(X,\beta),\prod_i\text{ev}_i^* E_i\otimes L_i^{l_i}\otimes\ca{O}^{\text{vir}}\otimes \ca{D}^{R,l}\big).
\]
When $\epsilon=\infty$ and $l=0$, this definition recovers the $K$-theoretic Gromov-Witten invariants of $X$ defined in \cite{tonita2}.

In the orbifold setting, one can still define the quasimap graph space $\ca{QG}^\epsilon_{g,k}(X,\beta)$ (see \cite[\textsection 2.5.3]{ciocan2}). The definition of the determinant line bundle $\ca{D}^{R,l}$ over the graph space is straightforward. We choose a basis $\{\phi_a\}$ of $K^0(\bar{I}_\mu X)\otimes\bb{Q}$. Let $\{\phi^a\}$ be the dual basis with respect to the twisted pairing $(\cdot,\cdot)^{R,l}$ on $K^0(\bar{I}_\mu X)\otimes\bb{Q}$ given by 
\[
(u,v)^{R,l}:=\chi\big(\bar{I}_\mu X,u\otimes \bar{\iota}^*v\otimes\text{det}^{-l}(\bar{I}_\mu\ca{R})\big).
\]
Here $\bar{\iota}$ is the involution induced by $(x,g)\mapsto (x,g^{-1})$ and $\bar{I}_\mu\ca{R}$ is a vector bundle over $\bar{I}_\mu X$ such that the fiber over $(x,H)$, with $H\subset\text{Aut}\ x$, is the $H$-fixed subspace of $\ca{R}_x$. With all the notations understood, it is straightforward to adapt the definition of cohomological orbifold quasimap $\ca{J}^\epsilon$-function \cite[Def. 3.1]{ciocan2} to the (permutation-equivariant) $K$-theoretic setting.

\section{Adelic Characterization in quantum $K$-theory with level structure}\label{conechar}
In this section, we focus on quantum $K$-theory, i.e., $(\epsilon=\infty)$-quasimap theory. We first recall the (virtual) Lefschetz-Kawasaki's Riemann-Roch formula. This is the main tool in analyzing the poles of the $J$-function. We give an adelic characterization of points on the cone $\ca{L}_{S_\infty}^{R,l}$ in Theorem \ref{levelladele}. As an application of the adelic characterization, we prove that certain ``determinantal '' modifications of points on the cone $\ca{L}_{S_\infty}$ of level $0$ lie on the cone $\ca{L}_{S_\infty}^{R,l}$ of level $l$. This result will be used in the proof of the toric mirror theorem in Section \ref{mirrorandmock}. We assume in this section that $X$ is a smooth projective variety which is not necessarily a GIT quotient. The determinant line bundle $\ca{D}^{R,l}$ is defined as in Remark \ref{notrep}.

\subsection{Virtual Lefschetz-Kawasaki's Riemann-Roch formula}\label{kloc}
To understand the poles of the generating series of the permutation-equivariant quantum $K$-invariants, we recall the \emph{Lefschetz-Kawasaki's Riemann-Roch formula} in \cite{givental19}. 

Let $h$ be a finite order automorphism of a holomorphic orbibundle $E$ over a compact smooth orbifold $\ca{M}$. The (super)trace of $h$ on the sheaf cohomology $H^*(\ca{M},E)$ can be computed as an integral over the $h$-fixed point locus $I\ca{M}^h$ in the inertia orbifold $I\ca{M}$:
\begin{equation}\label{eq:krr}
\text{tr}_h\,H^*(\ca{M},E)=\chi^{fake}\bigg(I\ca{M}^h,\frac{\text{tr}_{\tilde{h}}\,E}{\text{tr}_{\tilde{h}}\,\wedge^* N_{I\ca{M}^h}^\vee}\bigg):=\int_{[I\ca{M}^h]}\text{td}(T_{I\ca{M}^h})\,\text{ch}\bigg(\frac{\text{tr}_{\tilde{h}}\,E}{\text{tr}_{\tilde{h}}\,\wedge^* N_{I\ca{M}^h}^\vee}\bigg).
\end{equation}
We explain the ingredients of this formula as follows. By definition, we can choose an atlas of local charts $U\rightarrow U/G(x)$ of $\ca{M}$. The local description of the inertia orbifold $I\ca{M}$ near $x\in\ca{M}$ is given by $[\coprod_{g\in G(x)}U^g/G(x)]$, where $U^g\subset U$ denotes the fixed point locus of $g$.
The automorphism $h$ can be lifted to an automorphism $\tilde{h}$ of the chart $U^g$. We denote by $(U^g)^{\tilde{h}}$ the fixed point locus of $\tilde{h}$ in $U^g$. Then the local description of the orbifold $I\ca{M}^h$ is given by $[\coprod_{g}(U^g)^{\tilde{h}}/G(x)]$. We refer to the connected components of $I\ca{M}^h$ as \emph{Kawasaki strata}.
Near a point $(x,[g])\in I\ca{M}^h$, the tangent and normal orbiford bundles $T_{I\ca{M}^h}$ and $N_{I\ca{M}^h}$ are identified with the tangent bundle and normal bundle to $(U^g)^{\tilde{h}}$ in $U$, respectively. In the denominator of the right side of (\ref{eq:krr}), $\wedge^* N_{I\ca{M}^h}^\vee:=\sum_{i\geq0}(-1)^i\wedge^iN_{I\ca{M}^h}^\vee$ is the K-theoretic Euler class of the normal bundle $N_{I\ca{M}^h}$. The trace bundle $\text{tr}_{\tilde{h}}\,F$ is the virtual orbifold bundle:
\[
\text{tr}_{\tilde{h}}\,F:=\sum_\lambda\lambda\, F_\lambda,
\]
where $F_\lambda$ are the eigen-bundles of $h$ corresponding to the eigenvalues $\lambda$.
Finally, td and ch denote the \emph{Todd class} and \emph{Chern character}. 

By choosing $h$ to be the identity map, we obtain the \emph{Kawasaki's Riemann-Roch formula} \cite{kawasaki} from (\ref{eq:krr}):
\begin{equation}\label{eq:orkrr}
\chi(\ca{M},E)=\chi^{fake}\bigg(I\ca{M},\frac{\text{tr}_g\,E}{\text{tr}_g\,\wedge^* N_{I\ca{M}}^\vee}\bigg).
\end{equation}
When $\ca{M}$ is no longer smooth, Tonita \cite{tonita1} proved a \emph{virtual} Kawasaki's formula: under the assumption that $\ca{M}$ has a perfect obstruction theory and admits an embedding into a smooth orbifold which has the resolution property, Kawasaki's formula still holds true if we replace the structure sheaves, tangent, and normal bundles in the formula by their virtual counterparts.
According to \cite{abramovich3}, the moduli stacks of stable maps to smooth projective varieties satisfy the assumptions of Tonita's theorem. In the next subsection, we apply the virtual Kawasaki's Riemann-Roch (KRR) formula to $\overline{\ca{M}}_{g,k}(X,\beta)/S_k$ to study the poles of the $J$-function.

\subsection{Adelic characterization}
In this subsection, we first recall the adelic characterization \cite{givental13} of the cone $\ca{L}_{S_\infty}$ in the level-0 permutation-equivariant quantum $K$-theory. Then we generalize it to describe points on the cone $\ca{L}_{S_\infty}^{R,l}$ of level $l$.

\subsubsection{The level-0 case}
In the level-0 case, i.e., Givental-Lee's quantum $K$-theory, the permutation-equivariant invariants are defined as
\[
\langle \mb{t}(L),\dots,\mb{t}(L)\rangle_{g,k,\beta}^{S_k}:=\chi\big(\overline{\ca{M}}_{g,k}(X,\beta)/S_k,\ca{O}_{\overline{\ca{M}}_{g,k}(X,\beta)}^{\text{vir}}\otimes_{i=1}^k \mb{t}(L_i)\big),
\]
where $\mb{t}(q)$ is a Laurent polynomial in $q$ with coefficients in $K^0(X)\otimes\bb{Q}$. This is a special case of Definition \ref{definquasimap} with $\epsilon$ sufficiently large and $l=0$. The $J$-function is defined as
\[\ca{J}_{S_\infty}(\mb{t}(q),Q):=1-q+\mb{t}(q)+\sum_a\sum_{(k,\beta)\neq(0,0),(1,0)}Q^\beta\bigg\langle\frac{\phi_a}{1-qL},\mb{t}(L),\dots,\mb{t}(L)\bigg\rangle_{0,k+1,\beta}^{S_k}\ \phi^a.
\]
Here $\{\phi_a\}$ and $\{\phi^a\}$ are basis of $K^0(X)\otimes\bb{Q}$ dual with respect to the Mukai pairing
\[
(\phi_a,\phi_b):=\chi(\phi_a\otimes \phi_b).
\]
Recall the definition of the loop space $\ca{K}$ from Section \ref{quasigraph}:
\[\ca{K}:=[K^0(X)\otimes\bb{C}(q)]\otimes\bb{C}[[Q]].\]
With respect to the symplectic form (\ref{eq:sympform}) with $l=0$, there is a Lagrangian polarization $\ca{K}=\ca{K}_+\oplus\ca{K}_-$, where $\ca{K}_+$ consists of Laurent polynomials in $q$ and $\ca{K}_-$ consists of reduced rational functions in $q$. The Givental's cone $\ca{L}_{S_\infty}\subset\ca{K}$ is defined as the image of $\ca{J}_{S_\infty}:\ca{K}_+\rightarrow\ca{K}$.

To study the poles of the series $\ca{J}_{S_\infty}(q)$, we apply the virtual KRR formula to the stack $\overline{\ca{M}}_{g,k+1}(X,\beta)/S_k$, where the symmetric group acts on the last $k$ markings. By the virtual KRR formula, each term in the $J$-function can be written as a summations of fake Euler characteristics over the Kawasaki strata. Note that the Kawasaki strata parametrize stable maps with prescribed automorphisms, i.e., equivalence classes of pairs $(C,f,h)$, where $(C,f)$ is a stable map to $X$ and $h$ is an automorphism of the map. Here, $h$ is allowed to permute the last $k$ markings, but it has to preserve the first marking (with the insertion $1/(1-qL)$). Denote by $\eta$ the eigenvalue of $h$ on the \emph{cotangent} line to the curve at the first marking. 

There are two types of Kawasaki strata. Over the Kawasaki strata with $\eta=1$, the input $\text{tr}_h(1/(1-qL))$ in the fake Euler characteristics becomes $1/(1-q\bar{L})$, where $1-\bar{L}$ is nilpotent. From the finite expansion
\[
\frac{1}{1-q\bar{L}}=\sum_{i\geq 0}\frac{q^i(\bar{L}-1)^i}{(1-q)^{i+1}},
\]
we see that the contributions to the $J$-function from the Kawasaki strata with $\eta=1$ have poles at $q=1$.

Over the Kawasaki strata where $\eta\neq 1$ is a primitive $m$-th root of unity, the insertion $\text{tr}_h(1/(1-qL))$ in the fake Euler characteristics becomes $1/(1-q\eta\bar{L})$, where $1-\bar{L}$ is nilpotent. By considering the finite expansion
\[
\frac{1}{1-q\eta \bar{L}}=\sum_{i\geq 0}\frac{(q\eta)^i(\bar{L}-1)^i}{(1-q\eta)^{i+1}},
\]
we see that they contribute terms with possible poles at the root of unity $\eta^{-1}$ to the $J$-function. We refer the reader to \cite{givental13} for a nice diagram cataloging all the strata.

For each primitive $m$-th root of unity $\eta$, we denote by $\ca{J}_{S_\infty}(\mb{t})_{(\eta)}$ the Laurent expansion of the $J$-function in $1-q\eta$ and regard it as an element in the loop space of power $Q$-series with vector Laurent series in $1-q\eta$ as coefficients:
\[\ca{K}^\eta:=K^0(X)\bigg[\frac{1}{1-q\eta},1-q\eta\bigg]\bigg]\otimes\bb{C}[[Q]].\]

The contributions from the untwisted sector of $\overline{\ca{M}}_{g,k}(X,\beta)/S_k$ in the virtual KRR formula are called the \emph{fake} $K$-theoretic Gromov-Witten (GW) invariants. 
More precisely, they are defined by
\[
\langle \mb{t}(L),\dots,\mb{t}(L)\rangle_{0,k,\beta}^{fake}:=\int_{[\overline{\ca{M}}_{0,k}(X,\beta)]^\vir}\prod_{i=1}^k \text{ch}(\mb{t}(L_i))\,\text{Td}(T^\vir),
\]
where $[\overline{\ca{M}}_{g,k}(X,\beta)]^\vir$ is the virtual fundamental class of the moduli space and $T^\vir$ is the virtual tangent bundle of $\overline{\ca{M}}_{g,k}(X,\beta)$. We define the $J$-function in the fake quantum $K$-theory by
\begin{align*}
&\ca{J}_{fake}:\ca{K}_+^1\rightarrow\ca{K}^1,\\
&\ca{J}_{fake}(\mb{t}(q),Q):=1-q+\mb{t}(q)+\sum_a\sum_{(k,\beta)\neq(0,0),(1,0)}\frac{Q^\beta}{k!}\phi^a\bigg\langle\frac{\phi_a}{1-qL},\mb{t}(L),\dots,\mb{t}(L)\bigg\rangle_{0,k+1,\beta}^{fake}.
\end{align*}
Here the input $\mb{t}(q)$ belongs to 
\[
\ca{K}^1_+:=K^0(X)[[1-q]]\otimes\bb{C}[[Q]].
\]
Denote by $\ca{L}_{fake}\subset\ca{K}^1$ the range of the series $\ca{J}_{fake}$. The negative space $\ca{K}^1_-$ of the polarization is spanned by $\phi^aq^k/(1-q)^{k+1},\,a=1,\dots,\text{dim}\, K^0(X)_\bb{Q},\,k=0,1,\dots.$ The input $\mb{t}(q)$ of $\ca{J}_{fake}$ can be obtained from the projection of $\ca{J}_{fake}$ to $\ca{K}^1_+$ along $\ca{K}^1_-$. 

In \cite{givental13}, Givental gives the following adelic characterization of the values of $\ca{J}_{S_\infty}(\mb{t})$:
\begin{theorem}[\hspace{-0.0001 cm}\cite{givental13}]\label{level0adele}
The values of $\ca{J}_{S_\infty}(\mb{t})$ are characterized by the following requirements:
\begin{enumerate}
\item $\ca{J}_{S_\infty}(\mb{t})$ has possible poles only at 0, $\infty$, and roots of unity;
\item $\ca{J}_{S_\infty}(\mb{t})_{(1)}$ lies on $\ca{L}_{fake}$;
\item for every primitive root of unity $\eta$ of order $m\neq1$,
\begin{align*}
\ca{J}_{S_\infty}(\mb{t})_{(\eta)}&(q^{1/m}/\eta)\in\sqrt{\frac{\lambda_{-1}(T_X)}{\lambda_{-1}(\Psi^mT_X)}}\\
\times\,&\emph{exp}\sum_{i\geq1}\bigg(\frac{\Psi^iT_X^\vee}{i(1-\eta^{-i}q^{i/m})}-\frac{\Psi^{im}T_X^\vee}{i(1-q^{im})}\bigg)
\ca{T}_m\big(\ca{J}_{S_\infty}(\mb{t})_{(1)}\big),
\end{align*}
where $\Psi^m$ is the $m$-th Adams operation on $K^0(X)_\bb{Q}$ which acts on line bundles as $L\mapsto L^m$, $\lambda_{-1}(E):=\sum_i(-1)^i\wedge^iE^\vee$ is the $K$-theoretic Euler class of a vector bundle $E$, and $\ca{T}_m(\mb{f})$ is the space described in Definition \ref{twistedtangent} below.
\end{enumerate}
\end{theorem}

We recall the following definition from \cite{tonita3}.
\begin{definition}\label{twistedtangent}
Let $\mb{f}$ be a point on $\ca{L}_{fake}$ and let $T(\mb{f})$ be the tangent space to $\ca{L}_{fake}$ at $\mb{f}$, considered as the image of a map $S(q,Q):\ca{K}^1_+\rightarrow\ca{K}$. We extend the Adams operations from $K^0(X)_\bb{Q}$ to $\ca{K}^1$ by $\Psi^m(q)=q^m$ and $\Psi^m(Q)=Q^m$. Let $\Psi^{\frac{1}{m}}$ be the inverse of $\Psi^{m}$, acting as $q\mapsto q^{1/m}$ and $Q\rightarrow Q^{1/m}$. We define the space $\ca{T}_m(\mb{f})$ to be the image of the conjugate of $S(q,Q)$:
\[
\Psi^m\circ S(q,Q)\circ\Psi^{\frac{1}{m}}:\ca{K}^1_+\rightarrow\ca{K}^1.
\]
\end{definition}
\begin{remark}
As explained in \cite[Remark 5.6]{tonita3}, the explicit operator in condition (3) of Theorem \ref{level0adele} can be written as a composition $\Box_\eta\,\Box_m^{-1}$. The definitions of the operators $\Box_\eta$ and $\Box_m$ are given in Proposition \ref{sumofeffects}.  \end{remark}

\subsubsection{The general case of level $l$}
Let $R$ be a vector bundle over $X$. According to Remark \ref{notrep}, the level structure $\ca{D}^{R,l}$ is defined as
\[
{\ca{D}}^{R,l}=\big(\text{det}\,\ca{R}_{k,\beta}\big)^{-l},
\]
where $\ca{R}_{k,\beta}:=R\pi_*(\text{ev}^*\ca{R})$ is the \emph{index bundle}, $\pi:\ca{C}\rightarrow\overline{\ca{M}}_{g,k}(X,\beta)$ is the universal curve, and $\text{ev}:\ca{C}\rightarrow X$ is the universal evaluation morphism. To state the adelic characterization theorem of the cone $\ca{L}_{S_\infty}^{R,l}$ in the permutation-equivariant quantum $K$-theory with level structure, we introduce the fake invariants of level $l$
\[
\langle \mb{t}(L),\dots,\mb{t}(L)\rangle_{0,k,\beta}^{fake,R,l}:=\int_{[\overline{\ca{M}}_{0,k}(X,\beta)]^\vir}\prod_{i=1}^k \text{ch}(\mb{t}(L_i))\,\text{Td}(T^\vir)\,\text{ch}(\ca{D}^{R,l}),
\]
and the fake $J$-function of level $l$
\[
\ca{J}^{R,l}_{fake}(\mb{t}(q),Q):=1-q+\mb{t}(q)+\sum_a\sum_{(k,\beta)\neq(0,0),(1,0)}\frac{Q^\beta}{k!}\phi^a\bigg\langle\frac{\phi_a}{1-qL},\mb{t}(L),\dots,\mb{t}(L)\bigg\rangle_{0,k+1,\beta}^{fake,R,l}.
\]
Here $\{\phi^a\}$ is the dual basis of $\{\phi_a\}$ with respect to the twisted pairing 
\[
(u,v)^{R,l}=\chi(u\otimes v\otimes(\text{det}\,\ca{R})^{-l}).\]
Denote by $\ca{L}_{fake}^{R,l}\subset\ca{K}^1$ the range of the series $\ca{J}_{fake}^{R,l}$. 
Based on the relationship \cite{givental4} between gravitational descendants and ancestors of fake quantum $K$-theory, one can show that $\ca{L}_{fake}^{R,l}$ is an overruled Lagrangian cone. We refer the reader to \cite[\textsection{3}]{givental2} for more details.

\begin{convention}\label{changepairing}
We will consider various twisted theories, in which the pairings are usually different. In particular, the dual bases $\{\phi^a\}$ which appear in the definitions of various $J$-functions may not be the same. To relate $J$-functions in different theories, we need to regard them as elements of the same loop space. This is achieved by rescaling the elements in loop spaces. For example, there is a rescaling map
\begin{align*}
\big(\ca{K}^1,\ (\,,&\,)^{R,l}\big)\rightarrow\big(\ca{K}^1,\ (\,,\,)\big),\\
E&\mapsto E\otimes(\text{det}\,\ca{R})^{-l/2},
\end{align*}
which identifies the loop space in fake quantum $K$-theory of level $l$ with that in fake quantum $K$-theory of level 0.
\end{convention}

One of our main results is the following adelic characterization of values of the big $J$-function in quantum $K$-theory of level $l$, generalizing Theorem \ref{level0adele}.
\begin{theorem}\label{levelladele}
The values of $\ca{J}^{R,l}_{S_\infty}(\mb{t})$ are characterized by the following requirements:
\begin{enumerate}
\item $\ca{J}^{R,l}_{S_\infty}(\mb{t})$ has possible poles only at 0, $\infty$, and roots of unity;
\item $\ca{J}^{R,l}_{S_\infty}(\mb{t})_{(1)}$ lies on $\ca{L}^{R,l}_{fake}$;
\item for every primitive root of unity $\eta$ of order $m\neq1$,
\begin{align*}
\ca{J}^{R,l}_{S_\infty}(\mb{t})_{(\eta)}&(q^{1/m}/\eta)\in\sqrt{\frac{\lambda_{-1}(T_X)}{\lambda_{-1}(\Psi^mT_X)}}\\
\times\,&\emph{exp}\sum_{i\geq1}\bigg(\frac{\Psi^iT_X^\vee}{i(1-\eta^{-i}q^{i/m})}-\frac{\Psi^{im}T_X^\vee}{i(1-q^{im})}\bigg)\,\ca{T}^{R,l}_m\big(\ca{J}^{R,l}_{S_\infty}(\mb{t})_{(1)}\big),
\end{align*}
where the space $\ca{T}^{R,l}_m(\mb{f})$ is defined as in Definition \ref{twistedtangent} but starting with a point $\mb{f}\in\ca{L}^{R,l}_{fake}$.
\end{enumerate}
\end{theorem}

\subsubsection{The proof of Theorem \ref{levelladele}}
We follow the proofs of \cite{givental2} and \cite{tonita3}.
The first requirement in Theorem \ref{levelladele} is obviously satisfied. For the second condition, we apply the virtual KRR formula to the $J$-function. Let $\widetilde{\mb{T}}(q)$ be the sum of $\mb{t}(q)$ and all contributions in $\ca{J}^{R,l}_{S_\infty}(\mb{t})$ which are regular at $q=1$. According to Proposition \ref{cutedge}, the level structure $\ca{D}^{R,l}$ splits ``correctly'' over nodal strata. With this understood, the following proposition follows from an argument identical to the one given in \cite[Proposition 5.2]{tonita3}.
\begin{proposition}
We have \[\ca{J}^{R,l}_{S_\infty}(\mb{t}(q))_{(1)}=\ca{J}^{R,l}_{fake}(\widetilde{\mb{T}}(q))
\] 
as elements in $\ca{K}^1$. In particular, it shows that $\ca{J}^{R,l}_{S_\infty}(\mb{t})_{(1)}$ lies on the cone $\ca{L}^{R,l}_{fake}$.
\end{proposition}

Before we move on to prove the third condition in Theorem \ref{levelladele}, we characterize the cone $\ca{L}^{R,l}_{fake}$ of fake quantum $K$-theory of level $l$ in terms of the cone $\ca{L}_{fake}$ of level 0. This characterization will be needed later in the proof of Theorem \ref{levelladele}. 

Note that the fake quantum $K$-theory is a version of twisted cohomological Gromov-Witten theory. The machinery of twisted cohomological Gromov-Witten invariants was introduced in \cite{coates}, and generalized in various directions in \cite{tseng2,tonita4}. Let $\ca{H}$ be the loop space of the cohomological GW theory of $X$
\[
\ca{H}:=H^{even}(X,\bb{C})[z^{-1},z]][[Q]].
\]
It is equipped with a natural symplectic form $\Omega_H$ on $\ca{H}$ given by
\[
\Omega_H\big(\mb{f},\mb{g}\big)=\text{Res}_{z=0}\big(\mb{f}(-z),\mb{g}(z)\big)dz,
\]
where the pairing $(\,,)$ is the Poincar\'e pairing. With respect to $\Omega_H$, there is a Lagrangian polarization $\ca{H}=\ca{H}_+\oplus\ca{H}_-$, where
\[
\ca{H}_+:=H^{even}(X,\bb{C})[[z]][[Q]],\quad\ca{H}_-:=\frac{1}{z}H^{even}(X,\bb{C})[z^{-1}][[Q]].
\]
Inside $\ca{H}$, one can define an overruled Lagrangian cone $\ca{L}_H$ by the image of the cohomological big $J$-function. Since we will not use the explicit description of $\ca{L}_H$ in this paper, we refer the reader to \cite{coates} for the basic definitions in the cohomological GW theory. 
\begin{convention}
Throughout this subsection, we identify $\ca{K}^1$ with $\ca{H}$ via the Chern character 
\begin{align*}
&\text{qch}:\ca{K}^1\rightarrow\ca{H},\\
&E\mapsto\text{ch}(E),\,q\mapsto e^z.
\end{align*}
Hence $K$-theoretic insertions (e.g., $\phi_a$ and $L$) in the correlators of twisted cohomological theories should be understood as their Chern characters (e.g., $\text{ch}\,\phi_a$ and $\text{ch}\,L$).
\end{convention}

By definition, the fake quantum $K$-invariants are obtained from the cohomological invariants by inserting the Todd class $\text{Td}(T^\vir)$ of the virtual tangent bundle $T^\vir$. According to \cite{coates3}, the virtual tangent bundle can be written as
\begin{equation}\label{eq:decompoftangent}
T^\vir=\pi_*(\text{ev}^*T_X-1)-\pi_*(L_{k+1}^\vee-1)-(\pi_*i_*\ca{O}_\ca{Z})^\vee
\end{equation}
in $K^0(\overline{\ca{M}}_{0,k}(X,\beta))$. Here, $i:\ca{Z}\rightarrow\ca{C}$ is the embedding of the nodal locus. The three parts correspond respectively to: (i) deformations of maps to $X$ of a fixed source curve, (ii) deformations of complex structure and configuration of markings, and (iii) smoothing the nodes. 

It is proved in \cite{coates2} that the cone $\ca{L}_{fake}$ (of level 0) is given explicitly in terms of $\ca{L}_H$:
\[
\text{qch}(\ca{L}_{fake})=\Delta\ca{L}_H,
\] 
where the \emph{loop group transformation} $\Delta$ is the Euler-Maclaurin asymptotics of the infinite product
\[
\Delta\sim\prod_{\text{Chern roots $x$ of $T_X$}}\prod_{r=1}^\infty\frac{x-rz}{1-e^{-x+rz}}.
\]
Here, $\Delta$ acts on $\ca{H}$ by the pointwise multiplication, and it is determined only by the Todd class of the first summand in the expression of $T^\vir$. The second and third summands in (\ref{eq:decompoftangent}) are respectively responsible for the changes of the dilaton shifts and the polarizations between $\ca{H}$ and $\ca{K}^1$. We refer to \cite{coates3} for the details.

The fake quantum $K$-invariants of level $l$ are obtained from those of level $0$ by inserting one more class $\text{ch}(\ca{D}^{R,l})$. Its effect on the Lagrangian cone is described in the following proposition. 
\begin{proposition}\label{faketotwist}
Under the identification $z=\emph{log}\,q$, we have the following identity 
\[
\ca{L}_{fake}^{R,l}=\emph{exp}\bigg(-l\bigg(\frac{\emph{ch}_2\,\ca{R}}{z}\bigg)\bigg)\ca{L}_{fake}
\]
in the loop space $\big(\ca{K}^1,\,(\,,\,)\big)$.
\end{proposition}
\begin{proof}
Recall that the level structure $\ca{D}^{R,l}$ is defined as a certain power of the determinant of the index bundle $\ca{R}_{k,\beta}=R\pi_*(\text{ev}^*\ca{R})$.
Note that
\begin{equation}\label{eq:ourtwist}
\text{ch}(\ca{D}^{R,l})=\text{exp}\big(-l\cdot\text{ch}_1(\ca{R}_{k,\beta})\big).
\end{equation}
According to \cite{coates}, the cone of a theory twisted by a general multiplicative characteristic class of the form
\[
\text{exp}\big(\sum_{i\geq0}s_i\,\text{ch}_i(\ca{R}_{k,\beta})\big)
\]
is obtained from the cone of the untwisted theory by applying the operator
\[
\text{exp}\bigg(\sum_{m,i\geq 0}s_{2m-1+i}\frac{B_{2m}}{(2m)!}\text{ch}_i(\ca{R})\cdot z^{2m-1}\bigg).
\]
Here the Bernoulli numbers $B_{2m}$ are defined by
\[
\frac{t}{1-e^{-t}}=1+\frac{t}{2}+\sum_{m\geq1}\frac{B_{2m}}{(2m)!}t^{2m},
\]
and the operator acts on $\ca{H}$ by the pointwise multiplication. For the twisting class (\ref{eq:ourtwist}), we have $s_1=-l$ and $s_i=0$ if $i\neq1$. By applying the above result, we obtain the corresponding loop group transformation:
$$\text{exp}\bigg(-l\bigg(\frac{\text{ch}_2(\ca{R})}{z}+\frac{\text{ch}_0(\ca{R})\cdot z}{12}\bigg)\bigg).$$
Note that the cone $\ca{L}_{fake}$, being overruled, is invariant under multiplication by functions of $z$. Therefore, we can ignore the second summand in the exponent of the above operator.\end{proof}

Now let us prove the third condition in Theorem \ref{levelladele}. Let $\eta\neq1$ be a primitive root of unity of order $m$. The Kawasaki strata in $\overline{\ca{M}}_{g,k+1}(X,\beta)/S_k$ which contribute terms with poles at $q=\eta^{-1}$ to the $J$-function are called the \emph{stem spaces} in \cite{givental2}. We give a brief description of stem spaces here, and we refer the reader to \cite[\textsection{8}]{givental2} for more details. Let $(C',f,h)$ be a point in these strata. Consider the unique maximal subcurve $C_+\subset C'$ containing the first marking where the $m$-th power $h^m$ acts as the identity. Here we also require that the nodes between components in $C_+$ are \emph{balanced}, i.e., we require the eigenvalues of $h$ on the two branches of a node in $C_+$ are inverse to each other. Hence the subcurve $C_+$ is a chain of $\bb{P}^1$, on which $h$ acts as multiplication by $\eta$. There are only two smooth points on $C_+$ which are fixed by $h$: the first marking on the first component, and one more on the last component. The second point is called the \emph{butt} in \cite{givental2}. The butt can be a regular point, a marking, or a node in $C'$. The automorphism $h$ acts on the cotangent space at the butt by $\eta^{-1}$. The other marked points and unbalanced nodes on $C_+$ are cyclically permuted by $h$. We denote by $C$ the quotient of $C_+$ by the $\bb{Z}_m$-symmetry generated by $h$. The quotient curve $C$ together with the induced quotient stable map is called a \emph{stem} in \cite{givental2}. Note that a stem curve can carry unramified marked points, coming from symmetric configurations of $m$-tuples of markings on the cover, or nodes, coming from $m$-tuples of symmetric nodes on the cover, where further components of $C'$, cyclically permuted by $h$, are attached. 

One of the key observations in \cite{givental2} is that the data $(C_+,C,f)$ also represents a stable map to the orbifold $X/{\bb{Z}_m}=X\times B\bb{Z}_m$ in the sense of \cite{chenruan} and \cite{abramovich4}. Therefore, the contributions with poles at $q=\eta^{-1}$ in the KRR formula for the $J$-function can be expressed as cohomological integrals over the moduli space of stable maps to $X\times B\bb{Z}_m$, twisted by the Todd classes of the traces of the virtual tangent and normal bundles of the Kawasaki strata, and the Chern class of the trace of the level structure $\ca{D}^{R,l}$. To be more precise, we introduce some notations first. Let $\overline{\ca{M}}^{X,\beta}_{0,k+2}(\eta)$ denote the stem space. It parametrizes stems of degree $\beta$, which are quotient maps by the $\bb{Z}_m$-symmetry generated by $g$. Here $g$ acts by $\eta$ and $\eta^{-1}$ on the cotangent lines at the first and last markings of the covering curve, respectively. The only markings on the covering curve fixed by $h$ are the first and last markings. Note that the stem space is a Kawasaki stratum in $\overline{\ca{M}}_{0,mk+2}(X,m\beta)$. According to \cite[Proposition 5]{givental2}, the stem space $\overline{\ca{M}}^{X,\beta}_{0,k+2}(\eta)$ is isomorphic to the moduli space $\overline{\ca{M}}^{X/\bb{Z}_m,\beta}_{0,k+2}(g,1,\dots,1,g^{-1})$ of stable maps to the orbifold $X/\bb{Z}_m$. Here, the sequence $(g,1,\dots,1,g^{-1})$ indicates the sectors where the evaluation maps land. We also consider the stem space $\overline{\ca{M}}^{X,\beta}_{0,k+1}(\eta)$ parametrizing stems whose butts are regular points. Similarly, we have an isomorphism between $\overline{\ca{M}}^{X,\beta}_{0,k+1}(\eta)$ and $\overline{\ca{M}}^{X/\bb{Z}_m,\beta}_{0,k+1}(g,1,\dots,1)$.

For simplicity, we denote the stem space $\overline{\ca{M}}^{X,\beta}_{0,k+1}(\eta)$ by $\overline{\ca{M}}$. Modelling on the contributions in the virtual KRR formula applied to the stack $\overline{\ca{M}}_{g,k+1}(X,\beta)/S_k$, we define the correlators in the stem theory of level $l$ by
\[
\bigg\langle\frac{\phi}{1-qL^{1/m}},\mb{t}(L),\dots,\mb{t}(L)\bigg\rangle_{0,k+1,\beta}^{stem,R,l}
:=\int_{[\overline{\ca{M}}]^\vir}\,\text{td}(T_{\overline{\ca{M}}})\text{ch}\bigg(\frac{\text{ev}_1^*\,\phi\cdot\prod_{i=2}^{k}\text{ev}_i^*\,\mb{t}(L_i)\cdot\text{tr}_g\,\ca{D}^{R,l}}{\big(1-qL_1^{1/m}\big)\cdot\text{tr}_g\big(\wedge^*N_{\overline{\ca{M}}}^\vee\big)}\bigg).
\]
Here $[\overline{\ca{M}}]^\vir$ is the virtual fundamental class of the moduli space $\overline{\ca{M}}^{X/\bb{Z}_m,\beta}_{0,k+1}(g,1,\dots,1)$ of stable maps to $X/\bb{Z}_m$, and $T_{\overline{\ca{M}}}$ and $N_{\overline{\ca{M}}}$ are, respectively, the virtual tangent and normal bundles to $\overline{\ca{M}}$, considered as a Kawasaki stratum in $\overline{\ca{M}}_{0,mk+1}(X,m\beta)$. The line bundle $L_1$ is formed by the cotangent spaces of stem curves at the first markings, while $L_1^{1/m}$ corresponds to the cotangent line bundle of the covering curves (see \cite[\textsection{7}]{givental2} for the explanation). From the definition, we see that the stem theory of level $l$ is a type of twisted cohomological GW theory of $X\times B\bb{Z}_m$.

Before we investigate the stem theory further, let us recall some basic facts about the GW theory of the orbifold $X\times B\bb{Z}_m$. In this case, the Lagrangian cone of the cohomological GW theory of $X\times B\bb{Z}_m$ is the product of $m$ copies of the Lagrangian cone of the GW theory of $X$. It lies inside the product of $m$ copies of the Fock space $\ca{H}$. We refer to each copy of the Lagrangian cone as a \emph{sector}. These sectors are labeled by elements of $\bb{Z}_m=\{1,g,\dots,g^{m-1}\}$.

The following proposition relates the Laurent expansion of $\ca{J}^{R,l}_{S_\infty}(\mb{t})$ at $q=\eta^{-1}$ to generating series in stem theory.
\begin{proposition}\label{identifystem}
Let $\delta\mb{t}(q)$ be the contributions in $\ca{J}^{R,l}_{S_\infty}(\mb{t})$ which are regular at $q=\eta^{-1}$, i.e., $\delta\mb{t}(q)=1-q+\mb{t}(q)+\tilde{\mb{t}}(q)$, where $\tilde{\mb{t}}(q)$ is the sum of all the contributions from Kawasaki strata $\overline{\ca{M}}^{X,\beta}_{0,k+1}(\xi)$ with $\xi\neq \eta$. Then 
\[\ca{J}^{R,l}_{S_\infty}(\mb{t})_{(\eta)}=\delta{\mb{t}}(q)+\sum_a\sum_{(k,\beta)\neq(0,0)}\frac{Q^{m\beta}}{k!}\phi^a\bigg\langle\frac{\phi_a}{1-q\eta L^{1/m}},\mb{T}(L),\dots,\mb{T}(L),\delta{\mb{t}}(L^{1/m}/\eta)\bigg\rangle_{0,k+2,\beta}^{stem,R,l},\]
where
\begin{enumerate}
\item the evaluation morphisms at the marked points land in the twisted sector of $B\bb{Z}_m$ labeled by the sequence $(g,1,\dots,1,g^{-1})$,
\item $\mb{T}(L)=\Psi^m\,\widetilde{\mb{T}}(L)$, where $\Psi^{m}$ acts on cotangent line bundles $L\mapsto L^m$, elements of $K^0(X)_\bb{Q}$, and Novikov variables $Q^\beta\mapsto Q^{m\beta}$,
\item $\widetilde{\mb{T}}(q)$ is the input point of $\ca{J}^{R,l}_{S_\infty}(\mb{t})_{(1)}$, i.e., it is determined by
\[
1-q+\widetilde{\mb{T}}(q)=\big(\ca{J}^{R,l}_{S_\infty}(\mb{t})_{(1)}\big)_+.
\]
where $(\cdots)_+$ denotes the projection along $\ca{K}^1_-$ to $\ca{K}^1_+$.
\end{enumerate}
\end{proposition}

\begin{proof}
Proposition \ref{cutedge} shows that the determinant line bundle $\ca{D}^{R,l}$ factorizes 	``nicely'' over nodal strata. With this in mind, the argument of \cite[Proposition 2]{givental2} applies here with one slight change: when determining $\widetilde{\mb{T}}(q)$, we do not impose the same condition $\mb{t}(q)=0$ as in \cite[Proposition 2]{givental2}. This is because in the permutation-equivariant theory, we are allowed to permute marked points. 
\end{proof}

Proposition \ref{identifystem} shows that the Laurent expansion $\ca{J}^{R,l}_{S_\infty}(\mb{t})_{(\eta)}$ of the $J$-function around $q=\eta^{-1}$ can be identified with a tangent vector to the cone of stem theory of level $l$:
\begin{align*}
\delta&\ca{J}^{st,R,l}(\delta\mb{t},\mb{T}'):=\delta{\mb{t}}(q^{1/m})\\
&+\sum_a\sum_{(k,\beta)\neq(0,0)}\frac{Q^{\beta}}{k!}\phi^a\bigg\langle\frac{\phi_a}{1-q^{1/m} L^{1/m}},\mb{T}'(L),\dots,\mb{T}'(L),\delta{\mb{t}}(L^{1/m})\bigg\rangle_{0,k+2,\beta}^{stem,R,l},
\end{align*}
after replacing $q\eta$ with $q^{1/m}$ and $Q^\beta$ with $Q^{m\beta}$ (but not in $\delta{\mb{t}}$). Here the input point $\mb{T}'(q)$ is obtained from $\mb{T}(q)$ by replacing $Q^\beta$ with $Q^{\beta/m}$, and it belongs to the sector labeled by 1. The tangent vector $\delta\ca{J}^{st,R,l}(\delta\mb{t},\mb{T})$ belongs to the sector labeled by $g^{-1}$.

Now let us study the stem theory of level $l$ using the formalism of twisted cohomological GW theory of $X\times B\bb{Z}_m$. It follows from the definition that the stem theory of level $l$ is obtained from the untwisted cohomological Gromov-Witten theory of $X/B\bb{Z}_m$ by twisting the following classes:
\begin{equation}\label{eq:stemtwist1}
\text{td}(T_{\overline{\ca{M}}})/\text{ch}\big(\text{tr}_g\big(\wedge^*N_{\overline{\ca{M}}}^\vee\big)\big)
\end{equation}
and
\begin{equation}\label{eq:stemtwist2}
\text{ch}\big(\text{tr}_g\,\ca{D}^{R,l}\big).
\end{equation}
The trace in the first twisting class (\ref{eq:stemtwist1}) is computed in \cite[\textsection{8}]{givental2}, and the effects of this twisting class on the Lagrangian cone and the genus zero potential of the untwisted theory are also studied in \cite[\textsection{8}]{givental2}. We summarize them in the following proposition.
\begin{proposition}[\hspace{-0.0001 cm}\cite{givental2}]\label{sumofeffects}
The effects of the twisting class (\ref{eq:stemtwist1}) on the Lagrangian cone and the genus zero potential of the untwisted theory are described as follows:
\begin{enumerate}[(i)]
\item The sectors labeled by 1 and $g^{-1}$ are rotated by the operators $\Box_m$ and $\Box_\eta$, respectively. These two operators are the Euler-Maclaurin asymptotics of the infinite products
\begin{align*}
&\Box_m\sim\prod_i\bigg(\sqrt{\frac{x_i}{1-e^{-mx_i}}}\prod_{r=1}^\infty\frac{x_i-rz}{1-e^{-mx_i+rmz}}\bigg),\\
&\Box_\eta\sim\prod_i\bigg(\sqrt{\frac{x_i}{1-e^{-x_i}}}\prod_{r=1}^\infty\frac{x_i-rz}{1-\eta^{-r}e^{-x_i+rz/m}}\bigg),
\end{align*}
where $x_i$ are the Chern roots of the tangent bundle $T_X$.
\item
The dilaton shift changes from $-z$ to $1-q^m$.
\item There are changes of polarizations of symplectic loop spaces. More precisely, in the sector labeled by $1$, the negative space of the polarization is spanned by
\[
\phi^a\Psi^m(q^k/(1-q)^{k+1}),
\]
whereas in the sector labeled by $g^{-1}$, it is spanned by
\[
\phi^aq^{k/m}/(1-q^{1/m})^{k+1}.
\]
\end{enumerate}
\end{proposition}

We compute the second twisting class (\ref{eq:stemtwist2}), and describe its effect on the Lagrangian cone. Let $p$ be the universal family of stem curves. By abuse of notation, we still use $\text{ev}$ to denote the universal evaluation morphism from the universal family of quotient curves to $X/\bb{Z}_m$. Let $\bb{C}_{\eta^i}$ be the topologically trivial line bundle on $X/\bb{Z}_m$ on which $g$ acts as multiplication by $\eta^i$. According to a simple argument in \cite{tonita3}, the trace of the index bundle $R\pi_*(\text{ev}^*\ca{R})$ can be expressed as
\[
\text{tr}_g\big(R\pi_*(\text{ev}^*\ca{R})\big)=\sum_{i=0}^{m-1}\eta^i\,Rp_*(\text{ev}^*\,\ca{R}\otimes\bb{C}_{\eta^i}).
\]
For simplicity, we denote $Rp_*(\text{ev}^*\,\ca{R}\otimes\bb{C}_{\eta^i})$ by $\bar{\ca{R}}_i$. Then we have
\begin{align}
\text{ch}\,\text{tr}_g\,\ca{D}^{R,l}&=\text{ch}\,\bigg(\text{det}\,\sum_{i=0}^{m-1}\eta^i\,\bar{\ca{R}}_i\bigg)^{-l}\nonumber\\
&=\text{ch}\,\bigg(\prod_{i=0}^{m-1}\big(\eta^i\big)^{\text{ch}_0\,\bar{\ca{R}}_i }\prod_{i=0}^{m-1}\text{det}\,\bar{\ca{R}}_i\bigg)^{-l}\nonumber\\
&=\prod_{i=0}^{m-1} \text{exp}\bigg(-l\bigg(i\,\text{log}(\eta)\,\text{ch}_0\,\bar{\ca{R}}_i +\text{ch}_1\,\bar{\ca{R}}_i \bigg)\bigg)\label{eq:chdeter}
\end{align}

\begin{proposition}
Twisting by the class (\ref{eq:stemtwist2}) rotates the sector labeled by the identity of the Lagrangian cone of $X\times B\bb{Z}_m$ by 
\[
D_m:=\emph{exp}\bigg(-ml\bigg(\frac{\emph{ch}_2\,\ca{R}}{z}\bigg)\bigg)
\] 
The sector labeled by $g^{-1}$ is rotated by the same operator $D_m$.
\end{proposition}

\begin{proof}
The proof is based on the orbifold quantum Riemann-Roch theorem developed in \cite{tseng2}. Let $E$ be an orbifold vector bundle over $X\times B\bb{Z}_m$. Consider a general twisting class 
\[
\text{exp}\bigg(\sum_{j\geq0}s_j\,\text{ch}_j\,p_*\big(\text{ev}^*\,E\big)\bigg).
\]
According to \cite[Theorem 1]{tseng2}, it corresponds to the rotation by the following operator
\[
\text{exp}\bigg(\sum_{j\geq0}s_j\,\bigg(\sum_{n\geq0}\frac{(A_n)_{j+1-n}z^{n-1}}{n!}+\frac{\text{ch}_j\,E^{(0)}}{2}\bigg)\bigg).
\]
Here $A_n$ is an operator which acts on all sectors. The restriction $(A_n)|_{X,g^i}$ of $A_n$ to the sector labeled by $g^i$ is defined by 
\[
(A_n)|_{(X,g^i)}=\sum_{r=0}^{m-1}B_n\big(\frac{r}{m}\big)\text{ch}\, E^{(r)}_i,
\]
where $E_i^{(r)}$ (respectively $E^{(0)}$) is the subbundle of the restriction of $E$ to $(X,g^i)$ on which $g^i$ acts with eigenvalue $e^{2\pi i r/m}$ (respectively 1). The notation $(A_n)_j$ denotes the degree $j$ component of the operator $A_n$.
The Bernoulli polynomials are defined by
\[
\sum_{n\geq0}B_n(x)\frac{t^n}{n!}=\frac{te^{tx}}{e^t-1}.
\]

In our case, the twisting class is given by (\ref{eq:chdeter}). Let $\tilde{D}_{\eta,i}$ denote the symplectic transformations corresponding to the $i$-th factor of the twisting class (\ref{eq:chdeter}), restricted to $(X,g^{-1})$. For each $i\in\{0,\dots,m-1\}$, the operator $A_n$ in the definition of $\tilde{D}_{\eta,i}$ is given by
\[
(A_n)|_{(X,g^{-1})}=B_n\big(\frac{i}{k}\big)\,\text{ch}\,\ca{R}.
\]
By the orbifold quantum Riemann-Roch theorem, the operator $\tilde{D}_{\eta,i}$ equals
\[
\text{exp}\bigg(-l\,\text{log}(\eta)\bigg(\frac{\text{ch}_1\ca{R}}{z}+B_1\big(\frac{i}{m}\big)\,\text{ch}_0\ca{R}\bigg)-l\bigg(\frac{\text{ch}_2\ca{R}}{z}+B_1(\frac{i}{m})\,\text{ch}_1\,\ca{R}+\frac{B_2(i/m)\,\text{ch}_0\ca{R}}{2}z\bigg)-l\frac{\text{ch}_1\ca{R}}{2}\bigg)
\]
Let $\tilde{D}_\eta=\prod_{i=0}^{m-1}\tilde{D}_{\eta,i}$. To simplify the expression of $\tilde{D}_\eta$, we use the fact that $n\,\text{log}(\eta)=0$ if $n$ is an integer divisible by $m$. Keeping this in mind, we obtain 
\[
\tilde{D}_\eta=\text{exp}\bigg(-l\bigg(\frac{m\,\text{ch}_2\,\ca{R}}{z}+\frac{\text{ch}_0\,\ca{R}}{12\,m}z+\frac{\text{ch}_0\,\ca{R}}{6}\,\text{log}(\eta)\bigg)\bigg).
\]
Note that the factor $\text{exp}(-l(z\,\text{ch}_0\,\ca{R}/(12\,m)+\text{log}(\eta)\,\text{ch}_0\,\ca{R}/6))$ in $\tilde{D}_\eta$ is a scalar $z$-series and thus it preserves the overruled Lagrangian cone. We can drop it and obtain the operator $D_m$.

For the sector labeled by the identity, we denote by $\tilde{D}_{m,i}$ the restriction of the operator corresponding to the $i$-th factor of the twisting class (\ref{eq:chdeter}). It is easy to check that
\[
(A_n)|_{(X,1)}=B_n(0)\,\text{ch}\,\ca{R},
\]
and the operator $\tilde{D}_{m,i}$ equals
\[
\text{exp}\bigg(-il\,\text{log}(\eta)\bigg(\frac{\text{ch}_1\ca{R}}{z}\bigg)-ml\bigg(\frac{\text{ch}_2\ca{R}}{z}+\frac{\text{ch}_0\,\ca{R}}{12}z\bigg)\bigg).
\]
Let $\tilde{D}_m=\prod_{i=0}^{m-1}\tilde{D}_{m,i}$. Again by using the fact that $n\,\text{log}(\eta)=0$ if $m|n$, we can simplify the operator $\tilde{D}_m$ to
\[
\text{exp}\,\bigg(-ml\bigg(\frac{\text{ch}_2\ca{R}}{z}+\frac{\text{ch}_0\,\ca{R}}{12}z\bigg)\bigg).
\]
We can drop the second term in the exponent because it is a constant $z$-series.

\end{proof}

The above discussion can be summarized in the following proposition. 
\begin{proposition}\label{identifytangent}
$\emph{qch}\,\delta\ca{J}^{st,R,l}(\delta\mb{t},\mb{T})$ lies in the tangent space $\Box_\eta\,\Box^{-1}_m\big(\ca{T}_{\ca{I}^{tw}}\,\Box_m\,D_m\,\ca{L}_H\big)$ to the cone of the stem theory of level $l$ at a certain point $\ca{I}^{tw}$. The input $\mb{T}$ satisfies 
\[
\emph{qch}\big(1-q^m+\mb{T}(q)\big)=\big[\ca{I}^{tw}\big]_+,
\]
where $[\cdots]_+$ denotes the projection along the negative space of the polarization of the sector labeled by 1.
\end{proposition}
\begin{proof}
The argument of \cite[Proposition 5.8]{tonita3} applies here. We briefly explain the relation between the application point $\ca{I}^{tw}$ and the input $\mb{T}$. Note that the application point $\ca{I}^{tw}$ lies on the Lagrangian of the stem theory of level $l$ in the sector labeled by 1. According to Proposition \ref{sumofeffects} (ii), the new dilaton shift is $1-q^m$. This explains the equality in the proposition.
\end{proof}

To prove the third condition in Theorem \ref{levelladele}, we need to identify $\ca{T}_{\ca{I}^{tw}}\Box_mD_m\ca{L}_H$ with $\ca{T}_m\big(\ca{J}_{S_\infty}(\mb{t})_{(1)}\big)$. We first show that
\begin{proposition}\label{psim}
\[
\emph{qch}^{-1}\big(\Box_m\,D_m\,\ca{L}_H\big)=\widetilde{\Psi}^m\ca{L}_{fake}^{R,l}.
\]
Here, the Adams operation $\widetilde{\Psi}^{m}$ acts on $K$-theory classes of $X$ and $q$ by $\Psi^{m}(q)=q^m$, but not on the Novikov variables $Q$.
 
\end{proposition}
\begin{proof}
It is proved in \cite[Proposition 9]{givental2} that $\text{qch}^{-1}(\Box_m\,\ca{L}_H)=\Psi^m\ca{L}_{fake}$. In that proof, one needs to extend the action of the Adams operator $\Psi^m$ on cohomology classes via the Chern isomorphism:
\[
\text{ch}\big(\Psi^m(\text{ch}^{-1}\,a)\big)=m^{\text{deg}(a)/2}a.
\]
By Proposition \ref{faketotwist}, we have $\ca{L}_{fake}^{R,l}=D_1\ca{L}_{fake}$. We conclude the proof by noticing that $\Psi^m(D_1)=D_m$. 
\end{proof}

Let $\widetilde{\mb{T}}(q)$ be the input point of $\ca{J}^{R,l}_{fake}(\widetilde{\mb{T}}(q))$ determined by
\[
1-q+\widetilde{\mb{T}}(q)=\big(\ca{J}^{R,l}_{S_\infty}(\mb{t})_{(1)}\big)_+.
\]
where $(\cdots)_+$ denotes the projection along $\ca{K}^1_-$ to $\ca{K}^1_+$. Let $\mb{T}'(q)$ be the input point of $\ca{I}^{tw}$ such that
\begin{equation}\label{eq:comparepoints}
\widetilde{\Psi}^m\big(\ca{J}^{R,l}_{fake}(\widetilde{\mb{T}}(q))\big)=\ca{I}^{tw}\big(\mb{T}'(q)\big).
\end{equation}
We claim that $\widetilde{\Psi}^m(\widetilde{\mb{T}}(q))=\mb{T}'(q)$. This equality holds because according to Proposition \ref{sumofeffects} and Proposition \ref{psim}, the operation $\widetilde{\Psi}^m:\ca{K}^1\rightarrow\ca{K}^1$ identifies the cone $\ca{L}^{R,l}_{fake}$ with the cone $\Box_m\,D_m\,\ca{L}_H$, the polarization of the fake quantum $K$-theory with the polarization in the sector labeled by 1 of the stem theory, and the old dilaton shift $1-q$ with the new one $1-q^m$. Therefore $\widetilde{\Psi}^m$ must also map the input point $\widetilde{\mb{T}}(q)$ of the fake $J$-function to the input point $\mb{T}'(q)$ of $\ca{I}^{tw}$. Recall from Proposition \ref{identifystem} that we have $\Psi^m(\widetilde{\mb{T}}(q))=\mb{T}(q)$. Then it follows from the definitions of $\Psi^m$ and $\widetilde{\Psi}^m$ that $\mb{T}'(q)$ is obtained from $\mb{T}(q)$ by replacing $Q^\beta$ with $Q^{\beta/m}$.

By differentiate the relation (\ref{eq:comparepoints}), we get
\begin{align*}
\widetilde{\Psi}^m&\bigg(\mb{f}(q)+\sum\frac{Q^\beta}{k!}\phi^a\bigg\langle\frac{\phi_a}{1-qL},\widetilde{\mb{T}}(L),\dots,\widetilde{\mb{T}}(L),\mb{f}(L)\bigg\rangle_{0,k+2,\beta}^{fake,R,l}\bigg)\\
&=\widetilde{\Psi}^m\mb{f}(q)+\sum\frac{Q^\beta}{k!}\widetilde{\Psi}^m\phi^a\bigg\langle\frac{\widetilde{\Psi}^m\phi_a}{1-q^mL^m},\mb{T}'(L),\dots,\mb{T}'(L),\widetilde{\Psi}^m\mb{f}(L)\bigg\rangle_{0,k+2,\beta}^{stem,R,l}.
\end{align*}
The RHS is a tangent vector in $\ca{T}_{\ca{I}^{tw}}\Box_mD_m\ca{L}_H$ along the direction of $\delta\mb{t}':=\widetilde{\Psi}^m\mb{f}(q)$. The LHS becomes $\Psi^m\circ S(q,Q)\circ \Psi^{1/m}(\delta\mb{t}')$ after we replace $Q^\beta$ with $Q^{m\beta}$ (including such a change in $\widetilde{\mb{T}}$ but excluding it in $\mb{f}(q)$). This concludes the proof of Theorem \ref{levelladele}.

\begin{remark}\label{orbmap1}
When the target $X$ is an orbifold, the adelic characterization of points on the cone $\ca{L}$ of the ordinary, i.e., permutation-\emph{non}-equivariant, quantum $K$-theory is developed in \cite{tonita2}. In this case, the Lagrangian cone $\ca{L}$ has different sectors, and each sector corresponds to a connected component of the rigidified inertia stack $\bar{I}_\mu X$. Let $f:C\rightarrow X$ be an orbifold stable map. Here $C$ is an orbifold curve with possible orbifold structures at the marked points and nodes. Let $\underline{f}:\underline{C}\rightarrow\underline{X}$ be the map between coarse moduli spaces. There is a short exact sequence
\[
1\rightarrow K\rightarrow \text{Aut}(f)\rightarrow\text{Aut}(\text{\underline{f}})\rightarrow1.
\]
The kernel $K$ consists of automorphisms of $C\rightarrow X$ that fix $\underline{C}\rightarrow \underline{X}$. These automorphisms are referred to as ``ghost automorphisms'' in \cite{abramovich5}, and they arise from the stacky nodes of the source curve. 

To analyze poles of $K$-theoretic $J$-functions in the orbifold setting, we still apply the virtual KRR formula to the moduli space of orbifold stable maps. In this case, there are extra contributions from twisted sectors corresponding to ghost automorphisms. The key observation in \cite{tonita2} is that once we add the appropriate contributions from ghost automorphisms in the definition of fake $K$-theoretic GW invariants, the formalism of adelic characterizations carries over to the orbifold setting. We refer the reader to \cite[Definition 3.1]{tonita2} for the precise definition of fake $K$-theoretic invariants and \cite[Theorem 4.1]{tonita2} for the adelic characterization in the orbifold and permutation-non-equivariant setting. We only mention that if we restrict to the untwisted sector of the cone $\ca{L}$, the main theorem in \cite{tonita2} specializes to Theorem \ref{level0adele}.

The generalization of \cite[Theorem 4.1]{tonita2} to the permutation-equivariant setting is straightforward: we only need to change the application point of the tangent space in \cite[Definition 4.3]{tonita2} from $\ca{J}_1(0)$ to the Laurent expansion $(\ca{J}_{S\infty}(\mb{t}))_1$ of the $J$-function at $q=1$. Since the determinant line bundle splits ``correctly'' among nodal strata, we can also generalize \cite[Theorem 4.1]{tonita2} to permutation-equivariant quantum $K$-theory with level structure. In this paper, we focus on recovering examples of mock theta functions. For this purpose, we only need to consider the untwisted sector of Lagrangian cones of orbifold targets. Once we make this restriction, the statement of the adelic characterization is the same as in Theorem \ref{levelladele}. 
\end{remark}

\subsection{Determinantal modification}
In this subsection, we use the adelic characterization to prove Theorem \ref{detmodify} which gives us a way to obtain points on $\ca{L}^{R,l}_{S_\infty}$ by making certain ``determinantal'' modifications to points on the level-0 cone $\ca{L}_{S_\infty}$.

Let us restate Theorem \ref{detmodify}.
\begin{theorem}\label{detmodify1}
If \[
I=\sum_{\beta\in\emph{Eff}(X)}I_\beta Q^\beta
\]
lies on $\ca{L}_{S_\infty}$, then the point
\[
I^{R,l}:=\sum_{\beta\in\emph{Eff}(X)}I_\beta Q^\beta\prod_i\big(L_i^{-\beta_i}q^{(\beta_i+1)\beta_i/2}\big)^l
\]
lies on the cone $\ca{L}_{S_\infty}^{R,l}$ of permutation-equivariant quantum $K$-theory of level $l$. Here, $\emph{Eff}(X)$ denotes the semigroup of effective curve classes on $X$, $L_i$ are the $K$-theoretic Chern roots of $\ca{R}$, and $\beta_i:=\int_\beta\,c_1(L_i)$.
\end{theorem}

\begin{proof}
Suppose $I=\sum_{\beta\in\text{Eff}(X)}I_\beta Q^\beta$ is a point on $\ca{L}_{S_\infty}$. Let $I^{R,l}$ be its ``determinantal'' modification $\sum_{\beta}I_{\beta} Q^\beta\prod_i\big(L_i^{-\beta_i}q^{\beta_i(\beta_i-1)/2}\big)^l$. According to Convention \ref{changepairing}, we compare different cones in the same loop spare $\ca{K}^1$. In particular, $\ca{L}^{R,l}_{fake}$ and the tangent space in Theorem \ref{levelladele} are viewed as subspaces of $\ca{K}^1$. Therefore, to show $I^{R,l}$ lies on $\ca{L}^{R,l}_{\infty}$, we need to work with the series after the rescaling
$$\widetilde{I}^{R,l}:=(\text{det}\,\ca{R})^{-l/2} I^{R,l}.$$

We denote by $I_{(\eta)}$ and $\widetilde{I}^{R,l}_{(\eta)}$ the Laurent expansions of $I$ and $\widetilde{I}^{R,l}$ in $1-q\eta$, respectively. Let $Q_1,\dots,Q_n$ be the Novikov variables. Let $p_i$ be the degree 2 cohomology classes corresponding to $Q_i$, and let $P_i=e^{-p_i}\in K^0(X)$. We denote by $L_i$ the $K$-theoretic Chern roots of $\ca{R}$. In other words, we have $\text{ch}\,L_i=e^{l_i}$, where $l_i$ are cohomological Chern roots of $\ca{R}$. We write $l_i$ as a linear function $f_i(p_1,\dots,p_n)$ in terms of the basis $p_1,\dots,p_n$.

It is clear that $\widetilde{I}^{R,l}$ satisfies the first condition in Theorem \ref{levelladele}. Now, we check the second condition. It follows from the Lemma in the proof of \cite[Theorem 2]{coates} that the operator 
\[
\Phi:=\prod_{i=1}^n\text{exp}\bigg(l\bigg(\frac{\big(f_i(p_j-zQ_j\partial_{Q_j})\big)^2}{2z}+\frac{f_i\big(p_j-zQ_j\partial_{Q_j}\big)}{2}\bigg) \bigg)
\] preserves $\ca{L}_{fake}$. Define $d_j=\langle c_1(p_j), \beta\rangle$ to be the components of the degree $\beta$. By a simple computation, one can show that
\begin{align*}
\Phi(Q^\beta)&=\prod_{i=1}^n\text{exp}\bigg(l\bigg(\frac{\big(f_i(p_j-zd_j)\big)^2}{2z}+\frac{f_i\big(p_j-zd_j\big)}{2}\bigg) \bigg)Q^\beta\nonumber\\
&=\text{exp}\bigg(l\bigg(\frac{\text{ch}_2\,\ca{R}}{2z}+\frac{\text{ch}_1\,\ca{R}}{2}\bigg)\bigg)\,Q^\beta\nonumber\\
&\cdot\prod_{i=1}^n\text{exp}\bigg(l\bigg(\frac{\big(f_i(p_j-zd_j)\big)^2}{2z}+\frac{f_i\big(p_j-zd_j\big)}{2}\bigg) \bigg)\text{exp}\bigg(-l\bigg(\frac{\text{ch}_2\,\ca{R}}{2z}+\frac{\text{ch}_1\,\ca{R}}{2}\bigg)\bigg)\nonumber\\
&=\text{exp}\bigg(l\bigg(\frac{\text{ch}_2\,\ca{R}}{z}+\frac{\text{ch}_1\,\ca{R}}{2}\bigg)\bigg)\,Q^\beta\,\prod_i\big(L_i^{-\beta_i}q^{\beta_i(\beta_i-1)/2}\big)^l,
\end{align*}
It follows that
\begin{equation}\label{eq:usefulrel}
\Phi\big(\overline{Q}\cdot I_{(1)}\big)/\overline{Q}=\text{exp}\big(l\big(\text{ch}_2\,\ca{R}/z\big)\big) \widetilde{I}_{(1)}^{R,l},
\end{equation}
where $\overline{Q}:=\prod Q_i$. Since the LHS lies on $\ca{L}_{fake}$, we conclude that the Laurent expansion of $\widetilde{I}^{R,l}$ at $q=1$ lies on the cone $\ca{L}^{R,l}_{fake}=\text{exp}\big(-l\big(\text{ch}_2\,\ca{R}/z\big)\big)\ca{L}_{fake}$. 

Now we check the third condition. Suppose the tangent space to $\ca{L}_{fake}$ at $I_{(1)}$ is given as the image of a map
\[
S(q,Q):\ca{K}^1_+\rightarrow\ca{K}^1.
\]
Then by (\ref{eq:usefulrel}), the tangent space to $\ca{L}^{R,l}_{fake}$ at $\widetilde{I}^{R,l}_{(1)}$ is given as the image of a map
\[
S'(q,Q)=\text{exp}\bigg(-l\bigg(\frac{\text{ch}_2\,\ca{R}}{z}\bigg)\bigg)\Phi\circ S(q,Q):\ca{K}^1_+\rightarrow\ca{K}^1.
\]
Here we use the fact that the Novikov variables are contained in the $\lambda$-algebra and hence they preserve tangent spaces.

Recall from Definition \ref{twistedtangent} that the space $\ca{T}_m(I_{(1)})$ is defined as the image of a map
\[
\Psi^m\circ S(q,Q)\circ\Psi^{1/m}:\ca{K}^1_+\rightarrow\ca{K}^1.
\]
Then $\ca{T}_m^{R,l}(\widetilde{I}^{R,l}_{(1)})$ is given as the image of 
\begin{align*}
\Psi^m\circ S'(q,Q)\circ\Psi^{1/m}&=\Psi^m\circ\text{exp}\bigg(-l\bigg(\frac{\text{ch}_2\,\ca{R}}{z}\bigg)\bigg)\Phi\circ S(q,Q)\circ\Psi^{1/m}\\
&=\text{exp}\bigg(-ml\bigg(\frac{\text{ch}_2\,\ca{R}}{z}\bigg)\bigg)\Psi^m\circ\Phi\circ S(q,Q)\circ\Psi^{1/m}\\
&=D_m\,\Phi^m\,\big(\Psi^m\circ S(q,Q)\circ\Psi^{1/m}\big),
\end{align*}
where $\Phi^m:=\Psi^m(\Phi)$ is given as follows
\begin{align*}
\Psi^m(\Phi)=\prod_{i=1}^n\text{exp}\bigg(l\bigg(\frac{\big(f_i(mp_j-zQ_j\partial_{Q_j})\big)^2}{2mz}+\frac{f_i\big(mp_j-zQ_j\partial_{Q_j}\big)}{2}\bigg) \bigg).
\end{align*}
Here we use the fact that the Adams operation $\Psi^m$ acts on the degree two classes $z$ and $p_j$ as multiplication by $m$, and its action on the differential operator $zQ_j\partial Q_j$ is trivial\footnote{This is because $\Psi^m(zQ_j\partial Q_j)=mzQ_j^m\partial Q_j^m=zQ_j\partial Q_j$.}. This shows that $\ca{T}_m^{R,l}(\widetilde{I}^{R,l}_{(1)})=D_m\Phi^m\big(\ca{T}_m(I_{(1)})\big)$.

By the assumption, we have
\[
I_{(\eta)}(q^{1/m}/\eta)\in\Box_\eta\Box_m^{-1}\ca{T}_m(I_{(1)}).
\]
Then
\begin{align}
\widetilde{I}^{R,l}_{(\eta)}(q^{1/m}/\eta)&=\sum_{\beta}(I_\beta)_{(\eta)}(q^{1/m}/\eta) Q^\beta\prod_i\big(L_i^{-\beta_i}q^{(\beta_i+1)\beta_i/(2m)}\eta^{-(\beta_i+1)\beta_i/2}\big)^l\nonumber\\
&=\text{exp}\bigg(-ml\bigg(\frac{\text{ch}_2\,\ca{R}}{z}\bigg)\bigg)\big(\Phi(\overline{Q}\cdot I_{(\eta)})\big)(q^{1/m}/\eta)/\overline{Q}\label{eq:intercomp}.
\end{align}

By an elementary computation using the fact that $m\cdot\text{log}(\eta)=0$, one can show that for any series $\mb{f}$ in $q$ and $Q$, we have 
\begin{equation}\label{eq:claim1}
\big(\Phi(\mb{f})\big)\big(q^{1/m}/\eta\big)=\Phi^m\ca{D}_\eta\mb{f}(q^{1/m}/\eta),
\end{equation}
where the operator $\ca{D}_\eta$ is defined by
\[
\ca{D}_\eta=\prod_{i=1}^n\text{exp}\bigg(l\bigg(-f_i\big(\text{log}(\eta)(Q_j\partial Q_j)^2\big)+\frac{f_i\big(\text{log}(\eta)Q_j\partial Q_j\big)}{2}-\frac{m-1}{m}\frac{f_i\big(mp_j-zQ_j\partial_{Q_j}\big)}{2}\bigg) \bigg).
\]
Here the substitution $q\mapsto q^{1/m}/\eta$ corresponds to the change $z\mapsto z/m-\text{log}(\eta)$ in the expression of $\Phi$.

It follows from (\ref{eq:claim1}) that (\ref{eq:intercomp}) equals
\begin{equation}\label{eq:midstep1}
\big(D_m\Phi^m\ca{D}_\eta\big(\overline{Q}\cdot I_{\eta}(q^{1/m}/\eta)\big)\big)/\overline{Q}\in D_m\Phi^m\ca{D}_\eta\Box_\eta\Box_m^{-1}\ca{T}_m(I_{(1)}).
\end{equation}
Since all the operators above have constant coefficients (i.e. independent of $Q$), they commute. We claim that $\ca{D}_\eta$ preserves $\ca{T}_m(I_{(1)})$. This is because by definition, we have
\[
\ca{D}_\eta\ca{T}_m(I_{(1)})=\ca{D}_\eta\Psi^m\circ S(q,Q)\circ\Psi^{1/m}\ca{K}^1_+.
\]
Let 
\[
\ca{D}_{\eta,1}=\prod_{i=1}^n\text{exp}\big(l\big(-f_i\big(\text{log}(\eta)(Q_j\partial Q_j)^2\big)+f_i\big(\text{log}(\eta)Q_j\partial Q_j\big)/2\big)\big)
\]
and
\[\ca{D}_{\eta,2}=\prod_{i=1}^n\text{exp}\big(l\big(-(m-1)f_i\big(mp_j-zQ_j\partial_{Q_j}\big)/(2m)\big) \big)
\]
be the two factors of $\ca{D}_\eta$. Then it is easy to check that the first factor $\ca{D}_{\eta,1}$ commutes with $\Psi^m\circ S(q,Q)\circ\Psi^{1/m}$, and hence preserves $\ca{T}_m(I_{(1)})$. The second factor $\ca{D}_{\eta,2}$ satisfies the commutation relation:
\[
\ca{D}_{\eta,2}\Psi^m=\Psi^m\prod_{i=1}^n\text{exp}\big(l\big(-(m-1)f_i\big(p_j-zQ_j\partial_{Q_j}\big)/(2m)\big) \big).
\]
According to \cite[Corollary 1]{givental2}, the second operator on the RHS preserves the tangent space $S(q,Q)\circ\Psi^{1/m}\ca{K}^1_+$. Therefore we have shown that the space $\ca{T}_m(I_{(1)})$ is $\ca{D}_\eta$-invariant. 

We can further simplify the space on the RHS of (\ref{eq:midstep1}) as follows
\begin{align*}
D_m\Phi^m\ca{D}_\eta\Box_\eta\Box_m^{-1}\ca{T}_m(I_{(1)})&=D_m\Phi^m\Box_\eta\Box_m^{-1}\ca{D}_\eta\ca{T}_m(I_{(1)})\\
&=D_m\Phi^m\Box_\eta\Box_m^{-1}\ca{T}_m(I_{(1)})\\
&=D_m\Phi^m\Box_\eta\Box_m^{-1}(\Phi^m)^{-1}D_m^{-1}(\ca{T}_m^{R,l}(\widetilde{I}^{R,l}_{(1)}))\\
&=\Box_\eta\Box_m^{-1}(\ca{T}_m^{R,l}(\widetilde{I}^{R,l}_{(1)})).
\end{align*}
This concludes the proof.

\end{proof}
\begin{remark}\label{orbmap2}Suppose the target is an orbifold. As explained in Remark \ref{orbmap1}, the adelic characterization of points on the untwisted sector of the Lagrangian cone of $X$ is the same as the one given in Theorem \ref{levelladele}. Using the same proof as above, we can show that if $I$ is a point on the untwisted sector of $\ca{L}_{S_\infty}$, then the determinantal modification $I^{R,l}$ lies on the untwisted sector of $\ca{L}^{R,l}_{S_\infty}$.
\end{remark}

\section{Toric mirror theorem and mock theta functions}\label{mirrorandmock}
In this section, we first recall a mirror theorem proved by Givental \cite{givental15} for ordinary quantum $K$-theory of toric varieties. This corresponds to Theorem \ref{weakmirror} in the level 0 case. By combining Givental's result with Theorem \ref{detmodify} proved in the previous section, we obtain a toric mirror theorem for level-$l$ quantum $K$-theory. In some simple cases, we recover Ramanujan's mock theta functions from toric $I$-functions with nontrivial level structure.

\subsection{$K$-theoretic $I$-function and mock theta function}\label{computeifunc}
In this subsection, we first explicitly compute the (torus-equivariant) small $I$-functions with level structures for toric varieties, using quasimap graph spaces. Then we use torus localization to prove a toric mirror theorem (Theorem \ref{weakmirror}), following Givental \cite{givental15}. In the study of quantum $K$-theory with non-trivial level structures, a remarkable phenomenon is the appearance of Ramanujan's mock theta functions. 

Let $M\cong\bb{Z}^n$ be a $n$-dimensional lattice and let $N$ be its dual lattice. For every complete nonsingular fan $\Sigma\subset N_{\bb{R}}$, we can associate a $n$-dimensional smooth projective variety $X_\Sigma$. We denote by $\Sigma(1)$ the set of 1-dimensional cones in $\Sigma$.  Let $m=|\Sigma(1)|$. Each $\rho\in\Sigma(1)$ determines a Weil divisor $D_\rho$ on $X_\Sigma$ and the Picard group of $X_\Sigma$ is determined by the following short exact sequence:
\begin{equation}\label{eq:charge}
0\rightarrow M\rightarrow\bb{Z}^{\Sigma(1)}\rightarrow\text{Pic}(X_\Sigma)\rightarrow0.
\end{equation}
Here the inclusion is defined by $m\mapsto\sum_\rho\langle m,\rho\rangle D_\rho$. Now let us describe the quotient construction of $X_\Sigma$. Since $\text{Pic}(X_\Sigma)$ is torsion free, we choose an integral basis $\{L_1,\dots,L_s\}$ of it, where $s=m-n$. Then the inclusion map in (\ref{eq:charge}) is given by an integral $s\times n$ matrix $Q=(Q_{a\rho})$ which is called the charge matrix of $X_\Sigma$. Applying $\text{Hom}(-,\bb{C}^*)$ to the exact sequence (\ref{eq:charge}), we get an exact sequence.
\[
1\rightarrow G\rightarrow(\bb{C}^*)^{\Sigma(1)}\rightarrow N\otimes\bb{C}^*\rightarrow 1,
\]
where $G:=\text{Hom}(\text{Pic}(X_\Sigma),\bb{C}^*)\cong(\bb{C}^*)^s$. The first map in the above short exact sequence defines the following $G$-action on $\bb{C}^{\Sigma(1)}$
\begin{equation}\label{eq:act}
{\bf t}\cdot(z_{\rho_1},\dots,z_{\rho_m})=\bigg(\prod_{a=1}^st_a^{Q_{a\rho_1}}z_{\rho_1},\dots,\prod_{a=1}^st_a^{Q_{a\rho_m}}z_{\rho_m}\bigg),
\end{equation}
where ${\bf t}=(t_1,\dots,t_s)\in(\bb{C}^*)^s$.
By choosing an appropriate linearization of the trivial line bundle on $\bb{C}^{\Sigma(1)}$ (see e.g., \cite{dolgachev}, Chapter 12), the semistable and stable loci are equal. We denote this linearized trivial line bundle by $L_\Sigma$ and the stable loci by $U(\Sigma)$. Let $z_\rho$ be the coordinates in $\bb{C}^{\Sigma(1)}$. We define a subvarity 
\[
Z(\Sigma)=\{(z_\rho)\in\bb{C}^{\Sigma(1)}|\prod_{\rho\not\subset\sigma}z_\rho=0,\sigma\in\Sigma\}.
\]
Then we have
\[
U(\Sigma)=\bb{C}^{\Sigma(1)}\backslash Z(\Sigma).
\]
The toric variety $X_\Sigma$ is the geometric quotient $U(\Sigma)/G$. Let $P$ be the principal $G$-bundle $\bb{C}^{\Sigma(1)}\rightarrow[\bb{C}^{\Sigma(1)}/G]$. Let $\pi_i:G\rightarrow \bb{C}^*$ be the projection to the $i$-th component and let $R_j$ be the characters given by ${\bf t}=(t_1,\dots,t_s)\rightarrow\prod_{a=1}^st_a^{Q_{a\rho_j}}$ for $1\leq j\leq m$. Then the line bundles $L_i$ and $\ca{O}(-D_{\rho_j})$ are the restrictions of the associated line bundles of $P$ with the characters $\pi_i$ and $R_j$, respectively, to $X_\Sigma$.

Note that $X_{\Sigma}$ admits a $T^m:=(\bb{C}^*)^{\Sigma(1)}$-action. We denote by $P_i$ and $U_\rho$ the $T^m$-equivariant line bundles corresponding to $L_i$ and $\ca{O}(D_{\rho})$, respectively. In the $T^m$-equivariant $K$-group $K^0_{T^m}(X_\Sigma)\otimes\bb{Q}$, we have the following multiplicative relation:
\[
U_\rho=\prod_{i=1}^sP_i^{\otimes Q_{i\rho}}\Lambda^{-1}_\rho,\]
where $\Lambda_\rho$ are the generators of Repr($T^m$) corresponding to the projection to the component labeled by $\rho$.

Now let us compute the ($T^m$-equivariant) small $I$-function of $X_\Sigma$ with level structures, using the quasimap graph space. Let $\beta\in\text{Hom}_{\bb{Z}}(\text{Pic}^G(\bb{C}^{\Sigma(1)}),\bb{Z})$ be an $L_\Sigma$-effective class. According to \cite[Lemma 3.1.8]{ciocan3}, a point in the quasimap graph space $QG_{0,k}^{\epsilon=0+}(X_\Sigma,\beta)$ is specified by the following data
\[
((C,p_1,\dots,p_k),\{\ca{P}_i|i=1,\dots,s\},\{u_\rho\}_{\rho\in\Sigma(1)},\varphi),
\] 
where 
\begin{itemize}
\item $(C,p_1,\dots,p_s)$ is a connected, at most nodal, curve of genus $0$ and $p_i$ are distinct nonsingular points of $C$,
\item $\ca{P}_i$ are line bundles on $C$ of degree $f_i:=\beta(L_i)$,
\item $u_\rho\in\Gamma(C,\ca{L}_\rho)$, where $\ca{L}_\rho$ is defined by
\[
\ca{L}_\rho:=\otimes_{i=1}^s\ca{P}_i^{\otimes Q_{i\rho}},
\]
\item $\varphi:C\rightarrow\bb{P}^1$ is a regular map such that $\varphi_*[C]=[\bb{P}^1]$.
\end{itemize}
The stability conditions are discussed in Section \ref{quasigraph}. In the case when $(g,k)=(0,0)$, we have $C\cong\bb{P}^1$ and $\ca{P}_i\cong\ca{O}_{\bb{P}^1}(f_i)$. The line bundles $\ca{L}_\rho$ are isomorphic to $\ca{O}_{\bb{P}^1}(\sum_{i=1}^sf_iQ_{i\rho})=\ca{O}_{\bb{P}^1}(\beta_\rho)$, where $\beta_\rho:=\beta(\ca{O}(D_\rho))$. Therefore, a point on $QG_{0,0}^{\epsilon=0+}(X_\Sigma,\beta)$ is specified by sections $\{u_\rho\in\Gamma(\bb{P}^1,\ca{O}_{\bb{P}^1}(\beta_\rho))|\rho\in\Sigma(1)\}$. We choose coordinates $[x_0,x_1]$ on $\bb{P}^1$ and consider the standard action $\bb{C}^*$-action defined by (\ref{eq:actionst}). Let $F_0$ be the distinguished fixed point locus parametrizing quasimaps whose degrees are concentrated only at 0. According to \cite[\textsection{7.2}]{ciocan3}, we have the identification
\begin{align}\label{eq:idf0}
F_0&\cong \underset{\{\rho|\beta_\rho<0\}}{\bigcap} D_\rho\subset X_\Sigma\nonumber\\
(z_\rho\, x_0^{\beta_\rho})&\rightarrow(z_\rho),
\end{align}
where $(z_\rho)$ are the coordinates on $X_\Sigma$.

Let $R$ be a character of $G=(\bb{C}^*)^s$ defined by ${\bf t}\cdot z=\prod_{i=1}^st_i^{r_i}z$ where $r_i\in\bb{Z}$. Recall that the small $I$-function of $X_\Sigma$ of level $l$ and representation $R$ is defined by
\[
I^{R,l}(q)=1+\sum_{a}\sum_{\beta\neq0}Q^\beta\chi\bigg(F_0,\text{ev}^*(\phi_a)\otimes\bigg(\frac{\text{tr}_{\bb{C}^*}\ca{D}^{R,l}}{\text{tr}_{\bb{C}^*}\wedge^*\big(N_{F_0/QG}^{\text{vir}}\big)^\vee}\bigg)\bigg)\phi^a,
\]
It is not difficult to check that under the identification (\ref{eq:idf0}), we can identify the virtual normal bundle $N_{F_0/QG}^{\text{vir}}$ in $K^0(F_0)$ with
\begin{equation}\label{eq:normal}
N_{F_0/QG}^{\text{vir}}=\sum_{\{\rho|\,\beta_\rho>0\}}\sum_{i=1}^{\beta_\rho} \ca{O}(D_\rho)|_{F_0}\otimes\bb{C}_{-i}-\sum_{\{\rho|\,\beta_\rho<0\}}\sum_{i=1}^{\beta_\rho-1} \ca{O}(D_\rho)|_{F_0}\otimes\bb{C}_{i},
\end{equation}
where $\bb{C}_{a}$ denotes the representation of $\bb{C}^*$ on $\bb{C}$ with weight $a\in\bb{Z}$. Let $\ca{P}$ be the universal principal $G$-bundle on $F_0\times\bb{P}^1\subset X_\Sigma\times\bb{P}^1$. Then the associated line bundle $\ca{P}\times_G R$ can be identified with $\otimes_{i=1}^sL_i^{r_i}\otimes\ca{O}_{\bb{P}^1}(\beta_R)$, where $\beta_R:=\sum_{i=1}^sr_if_i$. We denote the line bundle $\otimes_{i=1}^sL_i^{r_i}$ by $\ca{R}$. Let $\pi:F_0\times\bb{P}^1\rightarrow F_0$ be the projection. When $\beta_R\geq0$, we have
\begin{align}
\ca{D}^{R,l}&=\text{det}^{-l}R\pi_*(\ca{R}\otimes\ca{O}_{\bb{P}^1}(\beta_R))\nonumber\\
&=\text{det}^{-l}(\ca{R}\otimes R^0\pi_*(\ca{O}_{\bb{P}^1}(\beta_R)))\nonumber\\
&=\ca{R}^{-l(\beta_R+1)}\otimes\bb{C}_{l\beta_R(\beta_R+1)/2}\label{eq:detbundle}.
\end{align}
When $\beta_R<0$, a similar calculation shows that we have the same formula $\ca{D}^{R,l}=\ca{R}^{-l(\beta_R+1)}\otimes\bb{C}_{l\beta_R(\beta_R+1)/2}$.

We give the explicit formulas of the (torus-equivariant) small $I$-functions of toric varieties in the following proposition.
\begin{proposition}\label{ifunc}
The small $I$-function of a toric variety $X_\Sigma$ of level $l$ and character $R$ is given by
\[
I^{R,l}(q)=1+\sum_{\beta\in\emph{Eff}(X)}Q^\beta\,\ca{R}^{-l\beta_R}\,q^{l\beta_R(\beta_R+1)/2}\prod_{\rho\in\Sigma(1)}\frac{\prod_{j=-\infty}^0(1-\ca{O}(-D_\rho)q^j)}{\prod_{j=-\infty}^{\beta_\rho}(1-\ca{O}(-D_\rho)q^j)},
\]
and its equivariant version is given by
\[
I^{R,l,eq}(q)=1+\sum_{\beta\in\emph{Eff}(X)}Q^\beta\,\tilde{\ca{R}}^{-l\beta_R}\,q^{l\beta_R(\beta_R+1)/2}\prod_{\rho\in\Sigma(1)}\frac{\prod_{j=-\infty}^0(1-U_\rho\,q^j)}{\prod_{j=-\infty}^{\beta_\rho}(1-U_\rho\,q^j)}.
\]
Here $\ca{R}:=\otimes_{i=1}^s{L_i}^{r_i}$ is the line bundle associated to the character $R$, and $\tilde{\ca{R}}=\otimes_{i=1}^s{P_i}^{r_i}$ and $U_\rho$ are the equivariant line bundles corresponding to $\ca{R}$ and $\ca{O}(-D_\rho)$, respectively.
\end{proposition}
\begin{proof}
The proposition follows easily from (\ref{eq:normal}) and (\ref{eq:detbundle}). Note that one factor $\ca{R}^{-l}$ in $(\ref{eq:detbundle})$ disappears due to the change of pairings (see Convention \ref{changepairing}). 
\end{proof}

By extending Givental's localization argument in \cite{givental15} to the setting with level structure, we prove a toric mirror theorem. Now let us restate Theorem \ref{weakmirror}.
\begin{theorem}\label{lamemirror}
Assume that $X_\Sigma$ is a smooth quasi-projective toric variety. Let $I^{R,l,eq}(q)$ be the level-$l$ torus-equivariant small $I$-function given in Proposition \ref{ifunc}. Then the series $(1-q)I^{R,l,eq}(q)$ lies on the cone $\ca{L}_{S_\infty}^{R,l,eq}$ in the symmetrized torus-equivariant quantum K-theory of level $l$ of $X_\Sigma$. \end{theorem}
\begin{proof}
For simplicity, we denote by $I$ the $I$-function $I^{R,l,eq}$ of $X_\Sigma$. Let $\{\phi_\alpha\}_{\alpha\in X^{T^m}_\Sigma}$ be the fixed point basis of $K^0_{T^m}(X_\Sigma)$ and let $\{\phi^\alpha\}$ be the dual basis with respect to the pairing (\ref{eq:twistedpairing}). For each fixed point $\alpha$, we denote by $J(\alpha)\subset \Sigma(1)$ the cardinality-$s$ subset such that $\alpha$ equals the intersection $\cap_{\rho\notin J(\alpha)}D_{\rho}$. Write $I=\sum_{\alpha}I^{(\alpha)}\phi_\alpha$. We denote by $U_\rho(\alpha)$ and $\ca{R}(\alpha)$ the restrictions of $U_\rho$ and $\ca{R}$ to the fixed point $\alpha$, respectively. For $\rho\in J(\alpha)$, we have $U_\rho(\alpha)=1$. Hence $I^{(\alpha)}$ can be explicitly written as 
\[
I^{(\alpha)}(q)=1+\sum_{\beta\in\text{Eff}'(X_\Sigma)}Q^\beta\frac{\ca{R}(\alpha)^{-l\beta_R}\,q^{l\beta_R(\beta_R+1)/2}}{\prod_{\rho\in J(\alpha)}\prod_{j=1}^{\beta{\rho}}(1-q^j)}\,\prod_{\rho\notin J(\alpha)}\frac{\prod_{j=-\infty}^0(1-U_\rho(\alpha)\,q^j)}{\prod_{j=-\infty}^{\beta_\rho}(1-U_\rho(\alpha)\,q^j)}.
\]
Here $\text{Eff}'(X_\Sigma)$ denotes the semigroup of effective curve classes $\beta$ such that $\beta_\rho\geq0$. The terms with $\beta_\rho<0$ disappear because there is a factor $(1-q^0)$ in the numerators.

We first observe that for the point target, the cone $\ca{L}_{S_\infty}^{pt,R,l}$ of level $l$ coincides with the cone $\ca{L}_\infty^{pt}$ of level 0. This is because in the case of the point target, the determinant line bundle $\ca{D}^{R,l}$ is always topologically trivial (with possible equivariant weights). Combining this observation with the fact that the level structure $\ca{D}^{R,l}$ splits ``nicely'' among nodal strata, we can extend the argument in \cite{givental12} to prove that a point $f=\sum_\alpha f^{(\alpha)}\phi_\alpha$ lies on $\ca{L}_{S_\infty}^{R,l,eq}$ if and only iff the following are satisfied:
\begin{enumerate}
\item When expanded as meromorphic functions with poles only at roots of unity, $f^{(\alpha)}$ lie on the cone $\ca{L}_{S_\infty}^{pt}$ in the permutation-equivariant quantum $K$-theory of the point target space.

\item Away from $q=0,\,\infty$, and roots of unity, $f^{(\alpha)}$ may have at most simple poles at $q= U_\rho(\alpha)^{-1/m},\,\rho\notin J(\alpha),\,m=1,2,\dots$, for generic values of $\Lambda_1,\dots,\Lambda_m$. The residues satisfy the following recursion relations
\[
\text{Res}_{q=U_\rho(\alpha)^{-1/m}}f^{(\alpha)}(q)\frac{dq}{q}=-\frac{\phi^{\alpha}Q^{md_{\alpha\rho}}}{C_{\alpha\rho}(m)}f^{(\rho)}\big(U_\rho(\alpha)^{-1/m}\big).
\]
Here $C_{\alpha\rho}(m)=\lambda_{-1}(T_p\,\overline{\ca{M}}_{0,2}(X_\Sigma,m{d_{\alpha\rho}}))\cdot\big(\ca{R}(\alpha)^{-md^R_{\alpha\rho}}\,q^{md^R_{\alpha\rho}(md^R_{\alpha\rho}+1)/2}\big)^l$, where
\begin{enumerate}
\item $T_p$ denotes the virtual tangent space to the moduli space at the point $p$ represented by the $m$-multiple cover of the one-dimensional orbit connecting $\alpha$ and $\rho$. The explicit formula of the equivariant weights of the $K$-theoretic Euler class $\lambda_{-1}(T_p)$ is given in \cite{givental15}.
\item $d_{\alpha\rho}$ denotes the degree of the one-dimensional orbit connecting $\alpha$ and $\rho$.
\item $d^R_{\alpha\rho}:=\langle d_{\alpha\rho},c_1(\ca{R})\rangle$.
\end{enumerate}
\end{enumerate}

We want to show that $(1-q)I$ satisfies (1) and (2). It is proved in the main theorem\footnote{In \cite{givental15}, the $I$-function is defined to sum over all $\beta\in\bb{Z}^m$. However, the same argument works if we restrict the summation to curve classes in the semigroup $\text{Eff}(X_\Sigma)$.} of \cite{givental15} that the series
\[
\widetilde{I}^{(\alpha)}=1+\sum_{\beta\in\text{Eff}'(X_\Sigma)}\frac{Q^\beta}{\prod_{\rho\in J(\alpha)}\prod_{j=1}^{\beta{\rho}}(1-q^j)}\,\prod_{\rho\notin J(\alpha)}\frac{\prod_{j=-\infty}^0(1-U_\rho(\alpha)\,q^j)}{\prod_{j=-\infty}^{\beta_\rho}(1-U_\rho(\alpha)\,q^j)}
\]
represents a value of $\ca{J}^{pt}_{S_{\infty}}(\mb{t}(q),Q)/(1-q)$, i.e., $(1-q)\widetilde{I}^{(\alpha)}$ lies on the cone $\ca{L}^{pt}_{S_\infty}$. Note that $I^{(\alpha)}$ is obtained from $\widetilde{I}^{(\alpha)}$ by a ``determinantal'' modification. Therefore, it follows from Theorem \ref{detmodify} that $(1-q)I^{(\alpha)}$ lies on $\ca{L}^{pt,R,l}_{S_\infty}=\ca{L}^{pt}_{S_\infty}$. 

To prove $(1-q)I^{(\alpha)}$ satisfies the second condition, we rewrite $\ca{R}(\alpha)^{-l\beta_R}\,q^{l\beta_R(\beta_R+1)/2}$ as
\begin{equation}\label{eq:rewritedet}
\frac{\prod_{j=-\infty}^{\beta_R}\big(\ca{R}(\alpha)^{-1}\,q^j\big)^l}{\prod_{j=-\infty}^{0}\big(\ca{R}(\alpha)^{-1}\,q^j\big)^l}.
\end{equation}
Note that for all $j$, we have
$
\ca{R}(\alpha)=\ca{R}(\beta)\lambda^{-d^R_{\alpha\rho}}$, where $\lambda=U_\rho(\alpha)$.
Hence at $q=\lambda^{-1/m}$, 
\[
\ca{R}(\alpha)^{-1}q^{j}=\ca{R}(\beta)^{-1}\,q^{j-md^R_{\alpha\rho}}.
\]
The formula (\ref{eq:rewritedet}) is equivalent to 
\[
\ca{R}(\alpha)^{-mld^R_{\alpha\rho}}\,q^{mld^R_{\alpha\rho}(md^R_{\alpha\rho}+1)/2}\frac{\prod_{j=-\infty}^{\beta_R-md_{\alpha\rho}^R}\big(\ca{R}(\beta)^{-1}\,q^j\big)^l}{\prod_{j=-\infty}^{0}\big(\ca{R}(\beta)^{-1}\,q^j\big)^l}
\]
at $q=\lambda^{-1/m}$. Combing the above equivalence with the result on page 10 of \cite{givental15}, we obtain
\begin{align*}
(1-q)I^{(\alpha)}(q)=\frac{Q^{md_{\alpha\rho}}}{1-q^m\lambda}\,\frac{\phi^\alpha}{C_{\alpha\rho}(m)}(1-q)I^{(\rho)}(q),
\end{align*}
which is equivalent to the residue formula in Condition (2).

Since $I$ is defined over the $\lambda$-algebra $\bb{Z}[\Lambda_1^{\pm},\dots,\Lambda_m^{\pm}][[Q]]$, it takes value in the symmetrized theory (see Remark \ref{symmetrized}).
\end{proof}

When $X_\Sigma$ is projective, we may pass to the non-equivariant limit in Theorem \ref{lamemirror} to obtain
\begin{corollary}
When $X_\Sigma$ is a smooth projective toric variety, the level-$l$ small $I$-function $I^{R,l}$ given in Proposition \ref{ifunc} lies on the cone $\ca{L}_{S_\infty}^{R,l}$ in the symmetrized quantum K-theory of level $l$.
\end{corollary}

We denote by St and St$^\vee$ the standard representation and its dual representation of $\bb{C}^*$. As corollaries of Theorem \ref{ifunc}, we give proofs for Proposition \ref{prop1}-\ref{prop3} and \ref{prop4}.
\begin{proof}[Proof of Proposition \ref{prop1}]
Let the target be $X=({\bb{C}}\backslash0)/{\bb{C}}^*$ where the action is the standard action. Then the Proposition follows directly from the Theorem \ref{ifunc}.
\end{proof}

\begin{proof}[Proof of Proposition \ref{prop2}]

Let the target be $X_{a_1,a_2}=({\bb{C}}^2\backslash\{(0,0)\})/{\bb{C}}^*$ with charge vector $(a_1,a_2)$. As mentioned in Remark \ref{firstorb}, Remark \ref{orbmap1} and Remark \ref{orbmap2}, we only consider the untwisted component of the orbifold $I$-function and its formula is given by Theorem \ref{ifunc}. 
\end{proof}

\begin{proof}[Proof of Proposition \ref{prop3}]For positive integers $a,b$, we consider $X_{a,-b}=\{(\bb{C}-0)\times\bb{C}\}/\bb{C}^*$ with charge vector $(a,-b)$. Let $\lambda,\mu$ be the generators of Repr($(\bb{C}^*)^2$) corresponding to the first and second projections of $(\bb{C}^*)^2$ onto its factors. From Theorem \ref{ifunc}, the untwisted component of the orbifold $I$-function is given by 
\begin{align*}I^{\text{St},\,l}_{X_{a,-b}}(q)&=1+\sum_{n\geq 1}\frac{ p^{nl}q^{\frac{n(n-1)l}{2}}(1-p^{-b}\mu^{-1})(1-p^{-b}\mu^{-1}  q^{-1})\cdots (1-p^{-b}\mu^{-1}  q^{1-bn})}{(1-p^a\lambda^{-1}  q)(1-p^a\lambda^{-1} q^2)\cdots (1-p^a\lambda^{-1}  q^{an})}Q^n\\
&=1+\sum_{n\geq 1}(-1)^{bn}\frac{ p^{nl-b^2n}q^{\frac{n(n-1)l-bn(bn-1)}{2}}\mu^{-bn}(1-p^{b}\mu)(1-p^b\mu  q^1)\cdots (1-p^b\mu  q^{bn-1})}{(1-p^a\lambda^{-1}  q)(1-p^a\lambda^{-1} q^2)\cdots (1-p^a\lambda^{-1}  q^{an})}Q^n.
\end{align*}
\end{proof}
\begin{proof}[Proof of Proposition \ref{prop4}]
We consider the target $O(-1)^{\oplus r}_{\bb{P}^{s-1}}=X_{\bf 1, -\bf 1}=\{({\bb{C}}^{s}-0)\times {\bb{C}}^r\}/{\bb{C}}^*$
with the charge vector $(1,1,\cdots, 1, -1, -1, \cdots, -1)$. It follows from Theorem \ref{ifunc} that 
\begin{align*}
I^{\text{St},\,l=1+s}_{X_{\bf 1, -\bf 1}}(q)&=1+\sum_{n\geq1}Q^np^{nl}q^{\frac{n(n-1)l}{2}}\frac{(p^{-1}\mu_1^{-1}, q)_{n}\cdots (p^{-1}\mu_r^{-1}; q)_{n}}
{(p\lambda_1^{-1}q; q)_{n}\cdots (p\lambda_s^{-1}q; q)_n}
\\
&=1+\sum_{n\geq1}(-1)^{nr}\prod_{i=1}^r(p\mu_i)^{-n}p^{(1+s)n}\frac{(p\mu_1, q)_{n}\cdots (p\mu_r; q)_{n}}
{(p\lambda_1^{-1}q; q)_{n}\cdots (p\lambda_s^{-1}q; q)_n}Q^n (q^{\frac{n(n-1)}{2}})^{1+s-r}.
\end{align*}
\end{proof}

\begin{remark}
Traditionally, there are two approaches to prove genus zero mirror theorems: an older approach of Givental-Tonita using adelic descriptions of cones and a more recent wall-crossing approach of Ciocan-Fontanine-Kim using quasimap theory. In the first version
of our paper, we use the wall-crossing approach. Recently, we discovered a gap in the proof. Roughly speaking, the $K$-theoretic version of the polynomiality property \cite[Lemma 7.6.1]{ciocan1} is not strong enough to determine the coefficients of the $K$-theoretic $S$-operators and $J$-functions recursively. In the current version, we switch back to Givental-Tonita's older technique. It is still an interesting question if we can improve the wall-crossing approach to prove the mirror theorem in the $K$-theory setting.

\end{remark}

\bibliographystyle{amsplain.bst}
\bibliography{reference}

\providecommand{\bysame}{\leavevmode\hbox to3em{\hrulefill}\thinspace}
\providecommand{\MR}{\relax\ifhmode\unskip\space\fi MR }
\providecommand{\MRhref}[2]{%
  \href{http://www.ams.org/mathscinet-getitem?mr=#1}{#2}
}
\providecommand{\href}[2]{#2}
\begin{thebibliography}{10}

\bibitem{abramovich5}
Dan Abramovich, Alessio Corti, and Angelo Vistoli, \emph{Twisted bundles and
  admissible covers}, Comm. Algebra \textbf{31} (2003), no.~8, 3547--3618,
  Special issue in honor of Steven L. Kleiman. \MR{2007376}

\bibitem{abramovich3}
Dan Abramovich, Tom Graber, Martin Olsson, and Hsian-Hua Tseng, \emph{On the
  global quotient structure of the space of twisted stable maps to a quotient
  stack}, J. Algebraic Geom. \textbf{16} (2007), no.~4, 731--751. \MR{2357688}

\bibitem{abramovich4}
Dan Abramovich, Tom Graber, and Angelo Vistoli, \emph{Algebraic orbifold
  quantum products}, Orbifolds in mathematics and physics ({M}adison, {WI},
  2001), Contemp. Math., vol. 310, Amer. Math. Soc., Providence, RI, 2002,
  pp.~1--24. \MR{1950940}

\bibitem{abramovich2}
\bysame, \emph{Gromov-{W}itten theory of {D}eligne-{M}umford stacks}, Amer. J.
  Math. \textbf{130} (2008), no.~5, 1337--1398. \MR{2450211}

\bibitem{okounkov2}
Mina Aganagic, Edward Frenkel, and Andrei Okounkov, \emph{Quantum
  {$q$}-{L}anglands correspondence}, Trans. Moscow Math. Soc. \textbf{79}
  (2018), 1--83. \MR{3881458}

\bibitem{okounkov3}
Mina Aganagic and Andrei Okounkov, \emph{Quasimap counts and {B}ethe
  eigenfunctions}, Mosc. Math. J. \textbf{17} (2017), no.~4, 565--600.
  \MR{3734654}

\bibitem{agnihotri}
{S}harad {A}gnihotri, \emph{{{Q}uantum cohomology and the {V}erlinde algebra}},
  Ph.D. Thesis, University of Oxford (1995).

\bibitem{behrend2}
Kai Behrend and Yuri~I. Manin, \emph{Stacks of stable maps and
  {G}romov-{W}itten invariants}, Duke Math. J. \textbf{85} (1996), no.~1,
  1--60. \MR{1412436}

\bibitem{belkale}
Prakash Belkale, \emph{Quantum generalization of the {H}orn conjecture}, J.
  Amer. Math. Soc. \textbf{21} (2008), no.~2, 365--408. \MR{2373354}

\bibitem{chenruan}
Weimin Chen and Yongbin Ruan, \emph{Orbifold {G}romov-{W}itten theory},
  Orbifolds in mathematics and physics ({M}adison, {WI}, 2001), Contemp. Math.,
  vol. 310, Amer. Math. Soc., Providence, RI, 2002, pp.~25--85. \MR{1950941}

\bibitem{ciocan2}
Daewoong Cheong, Ionu\c{t} Ciocan-Fontanine, and Bumsig Kim, \emph{Orbifold
  quasimap theory}, Math. Ann. \textbf{363} (2015), no.~3-4, 777--816.
  \MR{3412343}

\bibitem{ciocan3}
Ionu\c{t} Ciocan-Fontanine and Bumsig Kim, \emph{Moduli stacks of stable toric
  quasimaps}, Adv. Math. \textbf{225} (2010), no.~6, 3022--3051. \MR{2729000}

\bibitem{ciocan1}
\bysame, \emph{Wall-crossing in genus zero quasimap theory and mirror maps},
  Algebr. Geom. \textbf{1} (2014), no.~4, 400--448. \MR{3272909}

\bibitem{ciocan4}
Ionu\c{t} Ciocan-Fontanine, Bumsig Kim, and Davesh Maulik, \emph{Stable
  quasimaps to {GIT} quotients}, J. Geom. Phys. \textbf{75} (2014), 17--47.
  \MR{3126932}

\bibitem{coates3}
Tom Coates, \emph{Riemann-{R}och theorems in {G}romov-{W}itten theory}, PhD
  Thesis, University of California, Berkeley (2003). \MR{2705177}

\bibitem{coates2}
Tom Coates and Alexander Givental, \emph{Quantum cobordisms and formal group
  laws}, The unity of mathematics, Progr. Math., vol. 244, Birkh\"{a}user
  Boston, Boston, MA, 2006, pp.~155--171. \MR{2181805}

\bibitem{coates}
\bysame, \emph{Quantum {R}iemann-{R}och, {L}efschetz and {S}erre}, Ann. of
  Math. (2) \textbf{165} (2007), no.~1, 15--53. \MR{2276766}

\bibitem{dolgachev}
Igor Dolgachev, \emph{Lectures on invariant theory}, London Mathematical
  Society Lecture Note Series, vol. 296, Cambridge University Press, Cambridge,
  2003. \MR{2004511}

\bibitem{gepner1}
Doron Gepner, \emph{Fusion rings and geometry}, Comm. Math. Phys. \textbf{141}
  (1991), no.~2, 381--411. \MR{1133272}

\bibitem{givental3}
Alexander {Givental}, \emph{On the {WDVV} equation in quantum {$K$}-theory},
  Michigan Math. J. \textbf{48} (2000), 295--304. \MR{1786492}

\bibitem{givental4}
Alexander Givental, \emph{Symplectic geometry of {F}robenius structures},
  Frobenius manifolds, Aspects Math., E36, Friedr. Vieweg, Wiesbaden, 2004,
  pp.~91--112. \MR{2115767}

\bibitem{givental12}
Alexander {Givental}, \emph{{Permutation-equivariant quantum K-theory II. Fixed
  point localization}}, preprint (2015), arXiv:1508.04374.

\bibitem{givental13}
\bysame, \emph{{Permutation-equivariant quantum K-theory III. Lefschetz'
  formula on $\overline{\mathcal{M}}_{0,n}/S_n$ and adelic characterization}},
  preprint (2015), arXiv:1508.06697.

\bibitem{givental14}
\bysame, \emph{{Permutation-equivariant quantum K-theory IV.
  $\mathcal{D}_q$-modules}}, preprint (2015), arXiv:1509.00830.

\bibitem{givental15}
\bysame, \emph{{Permutation-equivariant quantum K-theory V. Toric
  $q$-hypergeometric functions}}, preprint (2015), arXiv:1509.03903.

\bibitem{givental16}
\bysame, \emph{{Permutation-equivariant quantum K-theory VI. Mirrors}},
  preprint (2015), arXiv:1509.07852.

\bibitem{givental17}
\bysame, \emph{{Permutation-equivariant quantum K-theory VII. General theory}},
  preprint (2015), arXiv:1510.03076.

\bibitem{givental18}
\bysame, \emph{{Permutation-equivariant quantum K-theory VIII. Explicit
  reconstruction}}, preprint (2015), arXiv:1510.06116.

\bibitem{givental11}
\bysame, \emph{Permutation-equivariant quantum {K}-theory {I}. {D}efinitions.
  {E}lementary {K}-theory of {$\overline{\mathcal{M}}_{0,n}/S_n$}}, Mosc. Math.
  J. \textbf{17} (2017), no.~4, 691--698. \MR{3734658}

\bibitem{givental19}
\bysame, \emph{{Permutation-equivariant quantum K-theory IX. Quantum
  Hirzebruch-Riemann-Roch in all genera}}, preprint (2017), arXiv:1709.03180.

\bibitem{givental21}
\bysame, \emph{{Permutation-equivariant quantum K-theory X. Quantum
  Hirzebruch-Riemann-Roch in genus 0}}, preprint (2017), arXiv:1710.02376.

\bibitem{givental22}
\bysame, \emph{{Permutation-equivariant quantum K-theory XI. Quantum
  Adams-Riemann-Roch}}, preprint (2017), arXiv:1711.04201.

\bibitem{givental2}
Alexander Givental and Valentin Tonita, \emph{The {H}irzebruch-{R}iemann-{R}och
  theorem in true genus-0 quantum {K}-theory}, Symplectic, {P}oisson, and
  noncommutative geometry, Math. Sci. Res. Inst. Publ., vol.~62, Cambridge
  Univ. Press, New York, 2014, pp.~43--91. \MR{3380674}

\bibitem{intriligator}
Kenneth Intriligator, \emph{Fusion residues}, Modern Phys. Lett. A \textbf{6}
  (1991), no.~38, 3543--3556. \MR{1138873}

\bibitem{jockers}
Hans {Jockers} and Peter {Mayr}, \emph{{A 3d Gauge Theory/Quantum K-Theory
  Correspondence}}, preprint (2018), arXiv:1808.02040.

\bibitem{willett}
Anton {Kapustin} and Brian {Willett}, \emph{{Wilson loops in supersymmetric
  Chern-Simons-matter theories and duality}}, preprint (2013), arXiv:1302.2164.

\bibitem{kawasaki}
Tetsuro Kawasaki, \emph{The {R}iemann-{R}och theorem for complex
  {$V$}-manifolds}, Osaka J. Math. \textbf{16} (1979), no.~1, 151--159.
  \MR{527023}

\bibitem{knudsen}
Finn~Faye Knudsen and David Mumford, \emph{The projectivity of the moduli space
  of stable curves. {I}. {P}reliminaries on ``det'' and ``{D}iv''}, Math.
  Scand. \textbf{39} (1976), no.~1, 19--55. \MR{0437541}

\bibitem{smirnov2}
Peter {Koroteev}, Petr~P. {Pushkar}, Andrey {Smirnov}, and Anton~M. {Zeitlin},
  \emph{{Quantum K-theory of Quiver Varieties and Many-Body Systems}}, preprint
  (2017), arXiv:1705.10419.

\bibitem{lee1}
Yuan-Pin Lee, \emph{Quantum {$K$}-theory. {I}. {F}oundations}, Duke Math. J.
  \textbf{121} (2004), no.~3, 389--424. \MR{2040281}

\bibitem{marian1}
Alina Marian and Dragos Oprea, \emph{Counts of maps to {G}rassmannians and
  intersections on the moduli space of bundles}, J. Differential Geom.
  \textbf{76} (2007), no.~1, 155--175. \MR{2312051}

\bibitem{marian2}
\bysame, \emph{The level-rank duality for non-abelian theta functions}, Invent.
  Math. \textbf{168} (2007), no.~2, 225--247. \MR{2289865}

\bibitem{marian3}
\bysame, \emph{G{L} {V}erlinde numbers and the {G}rassmann {TQFT}}, Port. Math.
  \textbf{67} (2010), no.~2, 181--210. \MR{2662866}

\bibitem{okounkov}
Andrei Okounkov, \emph{Lectures on {K}-theoretic computations in enumerative
  geometry}, Geometry of moduli spaces and representation theory, IAS/Park City
  Math. Ser., vol.~24, Amer. Math. Soc., Providence, RI, 2017, pp.~251--380.
  \MR{3752463}

\bibitem{smirnov3}
Andrei {Okounkov} and Andrey {Smirnov}, \emph{{Quantum difference equation for
  Nakajima varieties}}, arXiv e-prints (2016), arXiv:1602.09007.

\bibitem{smirnov1}
Petr~P. {Pushkar}, Andrey {Smirnov}, and Anton~M. {Zeitlin}, \emph{{Baxter
  Q-operator from quantum K-theory}}, preprint (2016), arXiv:1612.08723.

\bibitem{qu1}
Feng Qu, \emph{Virtual pullbacks in {$K$}-theory}, Ann. Inst. Fourier
  (Grenoble) \textbf{68} (2018), no.~4, 1609--1641. \MR{3887429}

\bibitem{RZ1}
Yongbin {Ruan} and Ming {Zhang}, \emph{{Verlinde/Grassmannian Correspondence
  and Rank 2 $\delta$-wall-crossing}}, preprint (2018), arXiv:1811.01377.

\bibitem{su}
Changjian {Su}, Gufang {Zhao}, and Changlong {Zhong}, \emph{{On the K-theory
  stable bases of the Springer resolution}}, preprint (2017), arXiv:1708.08013.

\bibitem{tonita4}
Valentin Tonita, \emph{Twisted orbifold {G}romov-{W}itten invariants}, Nagoya
  Math. J. \textbf{213} (2014), 141--187. \MR{3161407}

\bibitem{tonita1}
\bysame, \emph{A virtual {K}awasaki-{R}iemann-{R}och formula}, Pacific J. Math.
  \textbf{268} (2014), no.~1, 249--255. \MR{3207609}

\bibitem{tonita3}
\bysame, \emph{Twisted {K}-theoretic {G}romov-{W}itten invariants}, Math. Ann.
  \textbf{372} (2018), no.~1-2, 489--526. \MR{3856819}

\bibitem{tonita2}
Valentin {Tonita} and Hsian-Hua {Tseng}, \emph{{Quantum orbifold
  Hirzebruch-Riemann-Roch theorem in genus zero}}, arXiv e-prints (2013),
  arXiv:1307.0262.

\bibitem{tseng2}
Hsian-Hua Tseng, \emph{Orbifold quantum {R}iemann-{R}och, {L}efschetz and
  {S}erre}, Geom. Topol. \textbf{14} (2010), no.~1, 1--81. \MR{2578300}

\bibitem{vafa}
Cumrun Vafa, \emph{Topological mirrors and quantum rings}, Essays on mirror
  manifolds, Int. Press, Hong Kong, 1992, pp.~96--119. \MR{1191421}

\bibitem{witten1}
Edward Witten, \emph{The {V}erlinde algebra and the cohomology of the
  {G}rassmannian}, Geometry, topology, \& physics, Conf. Proc. Lecture Notes
  Geom. Topology, IV, Int. Press, Cambridge, MA, 1995, pp.~357--422.
  \MR{1358625}

\end{thebibliography}
\end{document}